\documentclass[11pt]{article}
\usepackage[centertags]{amsmath}
\usepackage{amsfonts}
\usepackage{amssymb}
\usepackage{amsthm,color}
\usepackage{newlfont}
\usepackage{bbm}

\newtheorem{Theorem}{Theorem}[section]
\newtheorem{Definition}[Theorem]{Definition}
\newtheorem{Proposition}[Theorem]{Proposition}
\newtheorem{Lemma}[Theorem]{Lemma}
\newtheorem{Corollary}[Theorem]{Corollary}
\newtheorem{Remark}[Theorem]{Remark}
\newtheorem{Example}[Theorem]{Example}

\newtheorem{Hypothesis}{Hypothesis}

\numberwithin{equation}{section}

\begin{document}

\def\le{\left}
\def\r{\right}
\def\cost{\mbox{const}}
\def\a{\alpha}
\def\d{\delta}
\def\ph{\varphi}
\def\e{\epsilon}
\def\la{\lambda}
\def\si{\sigma}
\def\La{\Lambda}
\def\B{{\cal B}}
\def\A{{\mathcal A}}
\def\L{{\mathcal L}}
\def\O{{\mathcal O}}
\def\bO{\overline{{\mathcal O}}}
\def\F{{\mathcal F}}
\def\K{{\mathcal K}}
\def\H{{\mathcal H}}
\def\D{{\mathcal D}}
\def\C{{\mathcal C}}
\def\M{{\mathcal M}}
\def\N{{\mathcal N}}
\def\G{{\mathcal G}}
\def\T{{\mathcal T}}
\def\R{{\mathbb R}}
\def\I{{\mathcal I}}

\def\bw{\overline{W}}
\def\phin{\|\varphi\|_{0}}
\def\s0t{\sup_{t \in [0,T]}}
\def\lt{\lim_{t\rightarrow 0}}
\def\iot{\int_{0}^{t}}
\def\ioi{\int_0^{+\infty}}
\def\ds{\displaystyle}
\def\pag{\vfill\eject}
\def\fine{\par\vfill\supereject\end}
\def\acapo{\hfill\break}

\def\beq{\begin{equation}}
\def\eeq{\end{equation}}
\def\barr{\begin{array}}
\def\earr{\end{array}}
\def\vs{\vspace{.1mm}   \\}
\def\rd{\reals\,^{d}}
\def\rn{\reals\,^{n}}
\def\rr{\reals\,^{r}}
\def\bD{\overline{{\mathcal D}}}
\newcommand{\dimo}{\hfill \break {\bf Proof - }}
\newcommand{\nat}{\mathbb N}
\newcommand{\E}{\mathbb E}
\newcommand{\Pro}{\mathbb P}
\newcommand{\com}{{\scriptstyle \circ}}
\newcommand{\reals}{\mathbb R}

\def\dist{{\textnormal{dist}}}
\def\Amu{{A_\mu}}
\def\Qmu{{Q_\mu}}
\def\Smu{{S_\mu}}
\def\H{{\mathcal{H}}}
\def\Im{{\textnormal{Im }}}
\def\Tr{{\textnormal{Tr}}}
\def\E{{\mathbb{E}}}
\def\P{{\mathbb{P}}}

\title{On a general approach to Freidlin-Wentzell exit problems for stochastic equations in Banach spaces}
\author{Michael Salins\\
\vspace{.1cm}\\
Department of Mathematics\\
 University of Maryland\\
College Park\\
 Maryland, USA
}

\date{}

\maketitle

\begin{abstract}
  Freidlin and Wentzell characterized the logarithmic asymptotics of the exit time from a basin of attraction for a finite dimensional diffusion with small noise. After that, several authors studied the same properties for exit problems associated to specific infinite dimensional systems. In this paper, we present a general method, based a control theoretic approach, to establish exit time and exit place results for a large class of stochastic equations in Banach spaces. 
\end{abstract}

\section{Introduction}

In \cite{fw}, Freidlin and Wentzell characterize the exit time and exit place asymptotics from a basin of attraction for stochastic differential equations of the form
\begin{equation} \label{finite-dim-eq}
  \left\{ \begin{array}{l}
    \ds{dX^\e_x(t) = f(X^\e_x(t)) dt + \sqrt{\e} \sigma(X^\e_x(t)) dW(t),} \\
    \vs
    \ds{X^\e_x(0) = x \in \mathbb{R}^d.}
  \end{array} \right.
\end{equation}
They investigate the asymptotics of the exit time,
\[\tau^\e_x = \inf\{t>0: X^\e_x(t) \not \in G \},\]
 where $G \subset \mathbb{R}^d$ is an open set such that the unperturbed system, $X^0_x$, is uniformly attracted to one asymptotically stable equilibrium point $a \in G$, without leaving $G$. Because the unperturbed system never leaves $G$, $\tau^\e_x$ will diverge as $\e \to 0$.
More specifically, using the theory of large deviations, it can be shown that this divergence is of exponential type and that for any $x \in G$,
\begin{equation} \label{eq:intro-finite-dim-exit}
  \lim_{\e \to 0} \e \log \E \tau^\e_x =  \inf_{y \in \partial G} V(y),
\end{equation}
where $V(y)$ is a non-negative function called the quasipotential.
Additionally, if there is a unique $y_0 \in \partial G$ such that $V(y_0) = \inf_{y \in \partial G} V(y)$, then
\[\lim_{\e \to 0} X^\e_x(\tau^\e_x) = y_0 \text{ in probability.}\]
In other words, $X^\e_x$ exits $G$ near $y_0$ with overwhelming probability.


In this paper, we deal with the following class of stochastic equations in a Banach space $E$,
\begin{equation} \label{intro-eq}
  \left\{ \begin{array}{l}
    \ds{dX^\e_x(t) = (AX^\e_x(t) + F(X^\e_x(t)) )dt + \sqrt{\e} B(X^\e_x(t)) dw(t),}\\
    \vs
    \ds{X^\e_x(0) = x \in E.}
  \end{array} \right.
\end{equation}
and we establish analogous exit time and exit place results. In the above equation, $A$ is the generator of a $C_0$ semigroup, and $F:E \to E$ is a dissipative nonlinear mapping. Typically $E$ would be a function space, for example, an $L^p$ space, the space of continuous functions, or a space of H\"older-continuous functions, and $A$ would be the realization of some linear differential operator.

The exit time asymptotics have previously been characterized for a variety of infinite dimensional equations including  stochastic reaction diffusion equations \cite{cerrok,cf-2011,cm-1997,f-1988}, stochastic damped Schr\"odinger equations \cite{gautier-2008}, stochastic semilinear wave equations \cite{sal}, and stochastic Navier-Stokes equations \cite{navier}. In each of these papers a proof of the exit time has been given, taking into account the specific structure of the underlying equation. Our aim in the current paper is to introduce a unified approach to study the exit problem for solutions of abstract stochastic evolution equations in Banach spaces that can apply to a wide variety of problems.

For a bounded open set $G \subset E$ that contains $0$, we study the exit times
\[\tau^\e_x = \inf\{t>0: X^\e_x(t) \not \in G \}.\]
We show that $\{X^\e_x\}_{\{\e>0\}}$ satisfies a large deviations principle on the space $C([0,T];E)$ with respect to the rate function $I_{0,T}$ given by
 \begin{equation} \label{eq:intro-I_0T}
   I_{0,T}(\varphi) = \frac{1}{2} \inf\left\{|\psi|_{L^2([0,T];H)}^2 : \varphi = X^\psi_x \right\}
 \end{equation}
 where $X^\psi_x$ is the unique mild solution to the deterministic control problem
 \[\left\{
 \begin{array}{l}
   \ds{\frac{d}{dt} X^\psi_x(t) = A X^\psi_x(t) + F(X^\psi_x(t)) + B(X^\psi_x(t)) \psi(t),} \\
   \vs
   \ds{X^\psi_x(0) = x,}
 \end{array}
 \right.\]
 and $H$ is some Hilbert space on which the noise $w(t)$ is defined.
 Importantly, we prove that this large deviation principle is uniform for initial conditions $|x|_E \leq R$, for each $R>0$. We recall the definition of a uniform large deviations principle.
 \begin{Definition} \label{def:uniform-LDP}
  The family $\{X^\e_x\}_{\e>0}$ in $C([0,T];E)$ satisfies a large deviations principle with speed $\e$ and rate function $I_{0,T}$ uniform with respect to initial conditions $|x|_E \leq R$, if for all $R>0$,
  \begin{enumerate}
    \item[\textit{i.}] For any $\psi \in L^2([0,T];H)$, $\delta>0$,
    \begin{equation} \label{eq:unif-LDP-lower}
      \liminf_{\e \to 0} \e \log \left( \inf_{|x|_E\leq R} \Pro \left|X^\e_x - X^\psi_x \right| < \delta \right) \geq -\frac{1}{2} |\psi|_{L^2([0,T];H)}^2.
    \end{equation}
    \item[\textit{ii.}] For any $r>0$ and $\delta>0$,
    \begin{equation} \label{eq:unif-LDP-upper}
      \limsup_{\e \to 0} \e \log \left( \sup_{|x|_E \leq R} \Pro \left( \dist_{C([0,T];E)}(X^\e_x, {K}^x_{0,T}(r)) > \delta \right) \right) \leq -r
    \end{equation}
    where
    \[{K}^x_{0,T}(r) = \left\{X^\psi_x \in C([0,T];E) : \frac{1}{2}|\psi|_{L^2([0,T];H)}^2 \leq r \right\}.\]
  \end{enumerate}
 \end{Definition}
Because $E$ is infinite dimensional, uniformity on bounded sets of initial conditions is a stronger requirement than uniformity for $x \in K$ where $K$ is a compact subset of $E$ (see for example \cite{bdm-2008,dz}). The set $\{x \in E: |x|_E \leq R\}$ is not compact if $E$ is infinite dimensional, but we still are able to prove a uniform large deviation principle.

Da Prato and Zabczyk \cite{DaP-Z,zab-87} characterized the exit problem for Hilbert space valued stochastic equations with dissipative nonlinearities and additive noise. They prove that
\begin{equation} \label{eq:intro-upper-Zabczyk}
  \limsup_{\e \to 0} \e \log \E \tau^\e_x \leq \bar{e}(\bar{G}^c)
\end{equation}
and
\begin{equation} \label{eq:intro-lower-Zabczyk}
 \liminf_{\e \to 0} \e \log \E \tau^\e_x \geq \underline{e}(\partial G),
\end{equation}
where
\[\bar{e}(\bar{G}^c) = \inf\{I_{0,T}(\varphi): \varphi(0) = 0, \varphi(T) \in \bar{G}^c , T>0\},\]
\[e_r(\partial G) = \inf\{I_{0,T}(\varphi): |\varphi(0)|_E = r, \varphi(T) \in \partial G, T>0\},\]
\[\underline{e}(\partial G) = \lim_{r \to 0} e_r(\partial G).\]
It is immediate from these definitions that $\underline{e}(\partial G) \leq \bar{e}(\bar{G}^c)$. In the finite dimensional case (such as \cite{fw}) one can show that $\underline{e}(\partial G) = \bar{e}(\bar{G}^c)$. In fact, Freidlin and Wentzell demonstrate that in this case, the function
\[e(x,y) = \inf\{I_{0,T}(\varphi): \varphi(0) = x, \varphi(T)=y, T>0\}\]
is actually continuous in $x$ and $y$. In the infinite dimensional case, however, one cannot expect such continuity because $e(x,y)=+\infty$ on a dense subset of $E \times E$. Despite this difficulty, in the current paper we show that $\underline{e}(\partial G) = \bar{e}(\bar{G}^c)$ for a large class of general Banach space valued problems. Such an equality guarantees that no gap exists between \eqref{eq:intro-lower-Zabczyk} and \eqref{eq:intro-upper-Zabczyk}.

 To prove this equality, we define $I_{-\infty,0}$ as an extension of the rate functions $I_{0,T}$ to the space $C((-\infty,0);E)$. This extension is possible because \eqref{intro-eq} is time homogeneous. We then define the quasipotential for any $N \subset E$ as
\[V(N) = \inf\{ I_{-\infty,0}(\varphi): \lim_{t \to -\infty} |\varphi(t)|_E = 0, \varphi(0) = N  \}.\]
The quasipotential satisfies
\[\underline{e}(\partial G) \leq V(\partial G) \leq \bar{e}(\bar{G}^c). \]
We will show that the above quantities are actually equal if the level sets of $I_{-\infty,0}$ are compact in the topology of uniform converge on bounded intervals, the map $x \in E \mapsto X^\psi_x(t) \in E$ is continuous uniformly in time, and the set $G$ satisfies the regularity assumption
\[V(\partial G) = V(\bar{G}^c).\]
We show that these assumptions are satisfied in unexpected generality. First, in section \ref{sec:linear} we study the case where \eqref{intro-eq} is linear and has additive noise. That is, $F(x)\equiv 0$ and $B(x)\equiv Q$ a linear operator. Using functional analytic arguments, we show that the compactness of the level sets of $I_{-\infty,0}$ is actually a consequence of the existence of $E$-valued mild solutions to \eqref{intro-eq}. In section \ref{sec:additive}, we show that if the noise is additive and the the nonlinearity $F$ satisfies a certain dissipativity assumption, then the level sets of $I_{-\infty,0}$ are still compact. The situation is significantly more complicated in the multiplicative noise case, but in section \ref{sec:multiplicative}, we provide some sufficient conditions that guarantee that the level sets of $I_{-\infty,0}$ are still compact.

Once we establish the compactness of the level sets of $I_{-\infty,0}$, we show that $\tau^\e_x$ diverges exponentially as $\e \to 0$. Specifically, we prove that
\begin{equation} \label{eq:intro-E-tau}\lim_{\e \to 0} \e \log \E \tau^\e_x = V(\partial G),\end{equation}
and
\begin{equation} \label{eq:intro-tau}\lim_{\e \to 0} \e \log \tau^\e_x = V(\partial G) \text{ in probability}.\end{equation}
We also prove that $X^\e_x$ is likely to exit $G$ near the points that minimize $V$ on the boundary of $G$. This means that if $N \subset \partial G$ is closed and \[V(N)>  V(\partial G),\] then
\[\lim_{\e \to 0} \Pro \left( X^\e_x(\tau^\e_x) \in N \right) =0.\]
In particular, if there exists a unique $y^* \in \partial G$ such that \[V(y^*) = \inf_{x \in \partial G} V(x),\] then by setting $N = \{y \in \partial G: |y-y^*|_E > \delta\}$, for any $\delta$, we can show that
\[\lim_{\e \to 0 } X^\e_x(\tau^\e_x) = y^* \text{ in probability.}\]

Our proofs of the exit time and exit place results are largely based on the proofs of \cite[Chapter 5]{dz}, but important and nontrivial modifications have to be introduced to allow us to deal with the infinite dimensionality of the problem.
Notice that because $G$ is open and infinite dimensional, it is never compact. The famous proofs of the exit time results for finite dimensional systems (\cite{dz,fw}) took advantage of the fact that bounded sets of $\reals^d$ are compact. In the current paper, we show that the compactness of $G$ is unnecessary because the level sets of $I_{-\infty,0}$ are compact. Because of this, our results extend the results of Chenal and Millet \cite{cm-1997}, where they studied the exit time of the stochastic heat equation from bounded subsets of the H\"older space $E=C^\alpha([0,1])$ for $\alpha>0$. The results of this paper allow us characterize the exit time and exit place from bounded subsets of the space of continuous functions $E=C([0,1])$.

In Section \ref{sec:assumptions}, we introduce our notations and we identify our hypotheses. In sections \ref{sec:linear}, \ref{sec:additive}, and \ref{sec:multiplicative}, we prove that the level sets of $I_{-\infty,0}$ are compact. Finally, in section \ref{sec-exit-proofs} we prove the exit time and exit place results.

\section{Assumptions and Preliminaries} \label{sec:assumptions}
Let $\left(E, |\cdot|_E\right)$ be a Banach space. We denote by $E^\star$ the dual space of $E$ and use the notation $\left<x,x^\star\right>_{E,E^\star}$ to represent the duality. For any $T>0$ we denote by $C([0,T];E)$ the Banach space of continuous functions from $[0,T]$ to $E$ endowed with the norm
 \[|\varphi|_{C([0,T];E)} = \sup_{t \in [0,T]} |\varphi(t)|_E.\]
 We denote by $C((-\infty,0];E)$ the metric space of continuous functions from $(-\infty,0]$ to $E$ endowed with the topology of uniform convergence on compact subsets of $(-\infty,0]$. Recall that this is a metric space under the metric
  \[\rho(\varphi, \psi) = \sum_{n=1}^\infty 2^{-n} \left(\sup_{t \in [-n, -n+1]} |\varphi(t) - \psi(t)|_E \right).\]
  For any set $G \subset E$, the complement of $G$ is denoted as $G^c = E \setminus G$ and the closure of $G$ is denoted as $\bar{G}$.

 Let $H$ be a Hilbert space of square integrable functions endowed with the inner product $\left<\cdot, \cdot\right>_H$ and norm $|\cdot|_H$. For any $-\infty \leq t_1 < t_2 \leq +\infty$ we denote by $L^2((t_1,t_2);H)$ the Hilbert space endowed with the inner product
 \[\left< \psi, \phi\right>_{L^2((t_1,t_2);H)} = \int_{t_1}^{t_2} \left< \psi(s), \phi(s) \right>_H ds.\]

We study the following equation
\begin{equation} \label{eq}
  \left\{
  \begin{array}{l}
    \ds{dX^\e_x(t) = (A X^\e_x(t) +F(X^\e_x(t))) dt + \sqrt{\e} B(X^\e_x(t)) dw(t),}\\
    \vs
    \ds{X^\e_x(0) = x,}
  \end{array}
  \right.
\end{equation}
where $A: D(A) \subseteq E \to E$ is the generator of a $C_0$ semigroup $S(t)$, $F:E \to E$ is a nonlinear mapping and $B$ maps $E$ into a space of linear mappings that are not necessarily bounded, but have the property that for any $t>0$ and $x \in E$, $S(t)B(x) \in \mathcal{L}(H,E)$. In the above equation, $w(t)$ is a cylindrical Wiener process on some Hilbert space $H$. This means that formally
\begin{equation} \label{eq:w-def}
  w(t) = \sum_{k=1}^\infty \beta_k(t) e_k,
\end{equation}
where $\{\beta_k\}$ is a family of independent one-dimensional Brownian motions on some stochastic basis $(\Omega, \mathcal{F}, \mathcal{F}_t, \Pro)$, and $\{e_k\}$ is a complete orthonormal system of $H$.

\begin{Definition}[Mild solution]
 An adapted process $X^\e_x \in C([0,T];E)$ is called a mild solution for \eqref{eq} if for $t \in [0,T]$,
 \begin{equation} \label{eq:mild-solution}
 X^\e_x(t) = S(t) x + \int_0^t S(t-s) F(X^\e_x(s)) ds + \sqrt{\e} \int_0^t S(t-s) B(X^\e_x(s)) dw(s).
 \end{equation}
\end{Definition}

We also introduce the deterministic control problem
\begin{equation} \label{eq:prelim-control}
  \left\{
  \begin{array}{l}
    \ds{\frac{d}{dt}X^\psi_x(t) = A X^\psi_x(t) + F(X^\psi_x(t)) + B(X^\psi_x(t)) psi(t),}\\
    \vs
    \ds{X^\psi_x(0) = x,}
  \end{array}
  \right.
\end{equation}
and its mild solution
\begin{equation} \label{eq:mild-control}
  X^\psi_x(t) = S(t)x + \int_0^t S(t-s)F(X^\psi_x(s))ds + \int_0^t S(t-s) B(X^\psi_x(s)) \psi(s) ds.
\end{equation}
\section{Linear equation with additive noise}\label{sec:linear}
We now consider the case of a linear stochastic equation with additive noise
\begin{equation} \label{eq:linear}
 \left\{
 \begin{array}{l}
   \ds{dX^\e_x(t) = A X^\e_x(t) dt + \sqrt{\e} Q dw(t),}\\
   \vs
   \ds{X^\e_x(0) = x \in E.}
 \end{array}
 \right.
\end{equation}
\begin{Hypothesis} \label{H:linear}
 The operator $A: D(A) \subset E \to E$ is the generator of a $C_0$ semigroup $S(t)$ of negative type. That is, there exist $M\geq1$ and $\omega>0$ such that $\|S(t)\|_{\mathcal{L}(E)} \leq M e^{-\omega t}$.
\end{Hypothesis}
Hypothesis \ref{H:linear} guarantees that the unperturbed equation $X^0_x$ has a globally asymptotic stable equilibrium at $0$.
The mild solution to \eqref{eq:linear} is given by
\begin{equation}
  X^\e_x(t) = S(t)x + \sqrt{\e} \int_0^t S(t-s)Q dw(s)=:S(t)x + \sqrt{\e}Z(t).
\end{equation}
We call $Z(t)$ the stochastic convolution.
\begin{Hypothesis} \label{H:QandZ}
  $Q$ is a linear operator with domain $H$. For any $t>0$, $S(t)Q \in \mathcal{L}(H, E)$. Furthermore, for any $t>0$, the stochastic convolution $Z(t)$ is an $E$-valued Gaussian random variable.
\end{Hypothesis}
\begin{Proposition}
  If  $E$ is a Hilbert space and for any $T>0$, and any complete orthonormal system $\{e_k\}$ of $H$,
  \begin{equation} \label{eq:suffient-cond-for-HQZ}
    \sum_{k=1}^\infty \int_0^T \left|S(t)Qe_k\right|_E^2 dt <+\infty,
  \end{equation}
  then Hypothesis \ref{H:QandZ} is satisfied.
\end{Proposition}
\begin{proof}
  We will show that under this assumption $\E|Z(t)|_E^2 <+\infty$. We assumed that $E$ was a Hilbert space, let $\{f_j\}$ be a complete orthonormal basis of $E$. Then for any $j$,
  \[\left<Z(t), f_j \right>_E = \sum_{k=1}^\infty \int_0^t \left<S(t-s)Q e_k, f_j \right>_E d\beta_k(s)\]
  which is a real-valued Gaussian random variable with variance
  \[\E \left< Z(t),f_j \right>_E^2 = \sum_{k=1}^\infty \int_0^t \left<S(t-s)Q e_k, f_j \right>^2 ds.\]
  It follows that
  \[\E |Z(t)|_E^2 = \E\sum_{j=1}^\infty \left<Z(t), f_j \right>_E^2 = \sum_{k=1}^\infty \int_0^t |S(t-s)Q e_k|_E^2 ds<+\infty \]
  and we conclude that $Z(t)$ is $E$-valued. 
\end{proof}
If $E$ is not a Hilbert space, sufficient conditions for satisfying Hypothesis \ref{H:QandZ} are not so obvious. In the next example we show that Hypothesis \ref{H:QandZ} is satisfied for a linear stochastic heat equation in one spatial dimension.
\begin{Example}[Stochastic heat equation in one spatial dimension]
  Consider the following linear stochastic heat equation for $\xi \in [0,\pi]$ and $t \in [0,+\infty]$,
  \begin{equation} \label{eq:stoch-heat}
   \left\{
   \begin{array}{l}
    \ds{\frac{\partial u}{\partial t}(\xi,t) = \frac{\partial^2 u}{\partial \xi^2}(\xi,t) + \frac{\partial w}{\partial t}(t,\xi)}\\
    \vs
    \ds{u(t,0)=u(t,\pi)=0, \ u(0,\xi) = u_0(\xi).}
   \end{array}
   \right.
  \end{equation}
  Let $E = C_0([0,\pi])$, the space of continuous real-valued functions from $[0,\pi]$ with boundary condtions $f(0)=f(\pi)=0$ for all $f \in E$, and let $H=L^2([0,\pi])$, the space of square integrable functions. The operator $\frac{\partial^2}{\partial \xi^2}$ generates a $C_0$ semigroup $S(t)$ called the heat semigroup on $E$. The spectrum of the operator $\frac{\partial^2}{\partial \xi^2}$ is the set $\{-k^2: k \in \nat\}$ and there exists a sequence of eigenfunctions
  \[e_k(\xi) := \frac{2}{\pi}\sin(k\xi), \ \frac{\partial^2 e_k}{\partial \xi^2}(\xi) = -k^2 e_k(\xi).\]
  It follows that for any $t\geq 0$, $k \in \nat$,
  \[S(t)e_k = e^{-k^2 t} e_k\]
  These eigenfunctions form a complete orthonormal system in $H$. The space-time white noise $\partial w/\partial t$ can be represented as the formal sum
  \[\partial w(\xi,t) = \sum_{k=1}^\infty e_k(\xi) d\beta_k(t).\]
  The solution to \eqref{eq:stoch-heat} is
  \[u(\xi,t) = S(t)u_0(\xi) + \int_0^t S(t-s) dw(s)(\xi).\]
  We can show that $u(\xi,t) \in E$ with probability 1 by following the arguments in \cite{DaP-Z}[Section 5.5].
\end{Example}
\begin{Remark}
  We notice that in the previous example, $Q = I$ and that the identity is not a continuous bounded operator from $H$ to $E$. However, because of the regularizing properties of the heat equation semigroup, $S(t) \in \mathcal{L}(H,E)$ for all $t>0$.
\end{Remark}
 In what follows, we will show that whenever Hypothesis \ref{H:QandZ} is satisfied, the operator $L \psi = \int_0^\infty S(s)Q\psi(s)ds$ is a compact operator.

Because the stochastic convolution $Z$ is Gaussian, it is known (see \cite[Theorems 12.4,12.5]{DaP-Z})  that $\{\sqrt{\e} Z\}_{\e >0} $ satisfies a large deviation principle with respect to the rate function
\[\tilde{I}_{0,T}(\varphi) = \frac{1}{2} \inf\{|\psi|_{L^2([0,T];H)}^2 : \varphi(t) = Z^\psi(t) \}\]
where
\[Z^\psi(t) = \int_0^t S(t-s) Q \psi(s) ds.\]
We use the standard convention that $\inf \emptyset = +\infty$. Then the following results holds.

\begin{Proposition} \label{prop:stoch-conv-LDP}
Under Hypotheses \ref{H:linear} and \ref{H:QandZ},
  \begin{enumerate}
    \item[\textit{i.}] For any $\psi \in L^2([0,T];H)$, and $\delta>0$,
    \begin{equation} \label{eq:stoch-conv-LDP-lower}
      \liminf_{\e \to 0} \e \log \left( \Pro \left|\sqrt{\e}Z - Z^\psi \right|_{C([0,T];E)} < \delta \right) \geq -\frac{1}{2} |\psi|_{L^2([0,T];H)}^2.
    \end{equation}
    \item[\textit{ii.}] For any $r>0$ and $\delta>0$,
    \begin{equation} \label{eq:stoch-conv-LDP-upper}
      \limsup_{\e \to 0} \e \log \left( \Pro \left( \dist_{C([0,T];E)}(\sqrt{\e} Z, \tilde{K}_{0,T}(r)) > \delta \right) \right) \leq -r
    \end{equation}
    where
    \[\tilde{K}_{0,T}(r) = \{Z^\psi \in C([0,T];E) : \frac{1}{2}|\psi|_{L^2([0,T];H)}^2 \leq r\}.\]
  \end{enumerate}
\end{Proposition}
From this LDP for the stochastic convolution we can immediately derive a uniform LDP for \eqref{eq:linear}.
Let $X^\psi_x$ solve the deterministic control problem
\begin{equation}
  \left\{
  \begin{array}{l}
    \ds{\frac{d}{dt}X^\psi_x(t) = AX^\psi_x(t) + Q \psi(t)}\\
    \vs
    \ds{X^\psi_x(0) = x.}
  \end{array}
  \right.
\end{equation}
That is
\[X^\psi_x(t) = S(t)x + Z^\psi(t).\]
Let $I_{0,T}: C([0,T];E) \to [0,+\infty]$ be given by
\[I_{0,T}(\varphi) = \inf\left\{\frac{1}{2} |\psi|_{L^2((0,T);H)}^2 : \varphi= X^\psi_x \right\}.\]
\begin{Proposition} \label{prop:linear-LDP}
Under Hypothesis \ref{H:linear} and \ref{H:QandZ}, the family $\{X^\e_x\}_{\e>0}$ satisfies a uniform LDP with respect to $I_{0,T}$. That is,
  \begin{enumerate}
    \item[\textit{i.}] For any $\psi \in L^2([0,T];H)$, $\delta>0$,
    \begin{equation} \label{eq:linear-LDP-lower}
      \liminf_{\e \to 0} \e \log \left( \inf_{x \in E} \Pro \left|X^\e_x - X^\psi_x \right|_{C([0,T];E)} < \delta \right) \geq -\frac{1}{2} |\psi|_{L^2([0,T];H)}^2.
    \end{equation}
    \item[\textit{ii.}] For any $r>0$ and $\delta>0$,
    \begin{equation} \label{eq:linear-LDP-upper}
      \limsup_{\e \to 0} \e \log \left( \sup_{x \in E} \Pro \left( \dist_{C([0,T];E)}(X^\e_x, {K}^x_{0,T}(r)) > \delta \right) \right) \leq -r
    \end{equation}
    where
    \[{K}^x_{0,T}(r) = \{X^\psi_x \in C([0,T];E) : \frac{1}{2}|\psi|_{L^2([0,T];H)}^2 \leq r\}.\]
  \end{enumerate}
\end{Proposition}
\begin{proof}
  This is a consequence of the fact that $X^\e_x(t) = S(t)x + \sqrt{\e}Z(t)$ and $X^\psi_x(t) = S(t)x + Z^\psi(t)$. Therefore, for any $ x \in E$,
  \[|X^\e_x - X^\psi_x|_{C([0,T];E)} = |\sqrt{\e}Z - Z^\psi|_{C([0,T];E)}.\]
\end{proof}
Now, we extend $I_{0,T}$ to the negative half line by translation, defining for any $\varphi \in C([-T,0];E)$
\[I_{-T,0}(\varphi) = I_{0,T}(\varphi(T + \cdot)),\]
\[I_{-\infty,0}(\varphi) = \sup_{T>0} I_{-T,0}(\varphi).\]
If $I_{-\infty,0}(\varphi)<+\infty$, then there exists $\psi \in {L^2((-\infty,0);H)}$ such that for $-T \leq t \leq 0$
\[\varphi(t) = S(t+T) \varphi(-T) + \int_{-T}^t S(t-s) Q \psi(s) ds.\]
If $\sup_{t \leq 0} |\varphi(t)|< +\infty$, then by taking $-T \to -\infty$,
\[\varphi(t) = \int_{-\infty}^t S(t-s) Q \psi(s) ds.\]
\begin{Definition}
  We define the quasipotential for $N \subset E$ as
  \begin{equation}
    \begin{array}{l}
    \ds{V(N):= \inf\left\{I_{-\infty,0}(\varphi): \lim_{t \to -\infty} |\varphi(t)|_E = 0, \ \varphi(0) \in N \right\}}\\
     \vs
    \ds{=\inf\left\{ \frac{1}{2} |\psi|_{L^2((-\infty,0);H)}^2 : \int_{-\infty}^0 S(-s)Q \psi(s) ds \in N \right\}.}
    \end{array}
  \end{equation}
\end{Definition}
We now prove that for any $r>0$, the level set of $I_{-\infty,0}$,
\begin{equation} \label{eq:linear-K}
  \mathcal{K}(r) = \left\{\varphi \in C((-\infty,0);E): \lim_{t \to-\infty}|\varphi(t)|_E = 0, \ I_{-\infty,0}(\varphi)\leq r  \right\},
\end{equation}
is compact in the topology of uniform convergence on bounded intervals of $(-\infty,0]$.

\begin{Definition} \label{def:L_t}
  Define for any $t \geq 0$, $L_t: L^2([0,t];H) \to E$
  \begin{equation}
    L_t \psi = \int_0^t S(s) Q \psi(s) ds.
  \end{equation}
\end{Definition}
\begin{Theorem} \label{thm:L_t-compact}
  For any fixed $t\geq0$, the operator $L_t:  L^2([0,t];H) \to E$ is a compact operator.
\end{Theorem}
\begin{proof}
  By Hypothesis \ref{H:QandZ}, $Z(t)$ is an $E$-valued Gaussian random variable. Therefore,
  \[\left<Z(t),x^\star\right>_{E,E^\star} = \sum_{k=1}^\infty \int_0^t \left<S(t-s)Q e_k, x^\star \right>_{E,E^\star} d\beta_k(s)\]
  is a real-valued Gaussian random variable. Its variance is
  \[\begin{array}{l}
  \ds{\E \left<Z(t),x^\star\right>_{E,E^\star}^2  = \sum_{k=1}^\infty \int_0^t \left<S(t-s)Q e_k, x^\star \right>_{E,E^\star}^2 ds}\\
  \vs
  \ds{= \sum_{k=1}^\infty \int_0^t \left<e_k, (S(t-s)Q)^\star x^\star \right>_H^2 ds = \int_0^t |(S(t-s)Q)^\star x^\star|_H^2 ds = |L_t^\star x^\star|_{L^2([0,t];H)}^2.}
  \end{array}\]
  We calculate the characteristic function of $Z(t)$. For any $x^\star \in E^\star$ let
  \begin{equation} \label{eq:characteristic-fct}
    \hat{\mu}(x^\star) := \E \left(e^{i \left<Z(t),x^\star \right>_{E,E^\star}} \right) = e^{-\frac{1}{2}|L_t^\star x^\star|_{L^2([0,t];H)}^2}.
  \end{equation}
  We now prove that $L_t^\star$ is a compact operator. Let $|x_n^\star|_{E^\star} \leq 1$ be a sequence in the dual space. By the Banach-Alaoglu theorem, the unit ball in $E^\star$ is weak-$\star$ compact and therefore there exists a subsequence $x_{n_k}^\star$ with the weak-$\star$ limit $x^\star$. Because $Z(t)$ is assumed to be $E$-valued with probability $1$, by the dominated convergence theorem,
  \[\lim_{k \to +\infty} \hat{\mu}(x_{n_k}^\star - x^\star) = \lim_{k \to +\infty} \E \left(e^{i \left<Z(t), x_{n_k}^\star - x^\star\right>_{E,E^\star}} \right) = 1.\]
  It follows from \eqref{eq:characteristic-fct} that
  \[\lim_{k \to +\infty}|L_t^\star(x_{n_k}^\star -  x^\star)|_{L^2([0,t];H)}^2 = 0\]
  and therefore $L_t^\star$ is a compact operator. By Schauder's Theorem (see, for example, \cite[Theorem VI.3.4]{conway}) $L_t: L^2([0,t];H) \to E$ is also a compact operator.
\end{proof}
We now recall a convenient fact about compact operators
\begin{Lemma} \label{lem:weak->strong}
  If $\mathcal{X}$ and $\mathcal{Y}$ are Banach spaces and $\Lambda: \mathcal{X} \to \mathcal{Y}$ is a compact operator, then for any sequence $\{x_n\} \subset E$ which converges weakly to $x$ in $\mathcal{X}$, we have $|\Lambda(x_n - x)|_{\mathcal{Y}} \to 0$.
\end{Lemma}
\begin{proof}
  It is clear that $\Lambda x_n \to \Lambda x$ weakly. This is because for any $x^\star \in \mathcal{X}^\star$,
  \[\left<\Lambda x_n, x^\star \right>_{\mathcal{X},\mathcal{X}^\star} = \left< x_n, \Lambda^\star x^\star\right>_{\mathcal{Y},\mathcal{Y}^\star} \to \left< \Lambda x, x^\star\right>_{\mathcal{X},\mathcal{X}^\star}.\]
  But the compactness of $\Lambda$ implies that every subsequence of $\Lambda x_n$ has a further subsequence that converges in $\mathcal{Y}$ norm. This limit must coincide with the weak limit $\Lambda x$.
\end{proof}
\begin{Corollary} \label{cor:L_t->0}
 For any $t>0$, the operator $L_t: L^2([0,t];H) \to E$ is bounded and
 \begin{equation}
   \lim_{t \to 0} \|L_t\|_{\mathcal{L}(L^2([0,t];H),E)} =0.
 \end{equation}
\end{Corollary}
\begin{proof}
  First, we notice that for $0<t<1$,
  \[L_{t} \psi = \int_0^t S(r)Q\psi(r) = \int_0^{1} S(r)Q \psi(r) \mathbbm{1}_{r < t} dr =  L_1( \psi(\cdot)\mathbbm{1}_{\cdot< t}). \]

  Let $\psi^t \in L^2([0,t];H)$ be such that $|\psi^t|_{L^2([0,t];H)} = 1$ and
  \begin{equation}
    |L_t \psi^t|_E > \|L_t\|_{\mathcal{L}(L^2([0,t];H),E)} - t.
  \end{equation}
  For any function $\phi \in L^2([0,1];H)$,
  \[\left|\int_0^t \left< \psi^t(s), \phi(s) \right>_H ds\right| \leq |\psi^t|_{L^2([0,t];H)}|\phi|_{L^2([0,t];H)} \stackrel{t \to 0}{\to} 0.  \]
  Therefore $\psi^t(\cdot) \mathbbm{1}_{\{\cdot \leq t\}}$ converges to $0$ weakly in $L^2([0,1];H)$.
  By Theorem \ref{thm:L_t-compact} and Lemma \ref{lem:weak->strong}.
  \[\|L_t\|_{\mathcal{L}(L^2([0,t];H),E)} \leq t+ |L_t \psi^t|_E = t + | L_1 \psi^t ( \mathbbm{1}_{\{\cdot< t\}})| \to 0\]
  Therefore, the conclusion follows.
\end{proof}
\begin{Remark}
  At first glance, Corollary \ref{cor:L_t->0} may seem obvious. In actuality, it is a consequence of compactness. There are certain noncompact operators $L_t$ for which $\|L_t\|_{\mathcal{L}(L^2([0,t];H),E)}$ does not converge to $0$. As an example, consider the operator $Q: H \to H$, $Qe_k = k e_k$ and the semigroup $S(t): H \to H$, $S(t)e_k = e^{-k^2t}e_k$. Then the operator
  \[L_t \psi = \int_0^t S(s)Q \psi(s)ds\]
  is a bounded operator, and for any fixed $\psi$,
   \[\lim_{t \to 0} |L_t \psi|_H = 0.\]
   However, $\|L_t\|_{\mathcal{L}(L^2([0,t]H),H)} = \frac{1}{\sqrt{2}}$ for all $t>0$. Of course, such an $L_t$ cannot be the quadratic variation of a Gaussian random variable because it is not compact.
\end{Remark}
\begin{Theorem} \label{thm:L-compact}
  The operator $L: L^2((0,+\infty,0);H) \to E$ defined by
  \begin{equation} \label{eq:L-def}
   L \psi = \int_0^\infty S(t) Q \psi(t) dt
  \end{equation}
  is compact.
\end{Theorem}
\begin{proof}
  This is a consequence of Theorem \ref{thm:L_t-compact} and the fact that $S(t)$ is of negative type. Notice that
  \[L \psi = \sum_{j=0}^\infty S(j) \int_{j}^{j+1} S(s-j) Q \psi(s) ds = \sum_{n=0}^\infty S(j) L_1 \psi(j+\cdot). \]
  For any sequence $|\psi_n|_{L^2((0,+\infty);H)} \leq 1$, there exists, by Alaoglu's theorem, a subsequence (which we relabel $\psi_n$) such that $\psi_n \to \psi$ weakly in $L^2((0,+\infty);H)$. Any translate $\psi_n(j+\cdot)$ also converges weakly to $\psi(j + \cdot)$. To see this, take any function $\phi \in L^2((0,+\infty);H)$ and notice that
  \[\int_0^\infty \left<\psi_n(j+s) - \psi(j+s), \phi(s) \right>_H ds = \int_0^\infty \left<\psi_n(s) - \psi(s), \phi(s-j) \chi_{\{s\geq j\}}\right>_H ds \to 0. \]
  Therefore, by Lemma \ref{lem:weak->strong}, for any $j$
  \[L_1 \psi_n(j+\cdot) \to L_1 \psi(j + \cdot).\]
  Our conclusion follows because
  \[\left|L(\psi_n -\psi)\right|_E \leq M\sum_{j=0}^\infty e^{-\omega j} |L_1(\psi_n(j+\cdot) - \psi(j+\cdot))|_E. \]
\end{proof}
\begin{Theorem} \label{thm:linear-K-compact}
  Under Hypothesis \ref{H:linear} and \ref{H:QandZ}, for any $r >0$, the set
  $\mathcal{K}(r)$ defined in \eqref{eq:linear-K}
  is compact in the topology of uniform convergence on compact sets of $(-\infty,0]$.
\end{Theorem}

\begin{proof}
  Given any sequence
  \[\Psi_n(t) = \int_{-\infty}^t S(t-s) Q \psi_n(s) ds\]
  with $\frac{1}{2} |\psi_n|_{L^2((-\infty,0))}^2 \leq r$,
  we will show that there exists a subsequence which is convergent in $C((-\infty,0);E)$.
  By a time change, we can write
  \[\Psi_n(t) = \int_0^\infty S(s)Q \psi_n(t-s) ds = L \psi_n(t-\cdot),\]
  where $L$ is defined in \eqref{eq:L-def}.
  By Alaoglu's theorem, there exists a subsequence, which we relabel $\psi_{n}$, and $\psi \in L^2((-\infty,0];H)$ such that
  \[\psi_{n} \to \psi \text{ weakly in } L^2((-\infty,0);H).\]
  Now, integrating against any function $\phi \in L^2((0,+\infty);H)$ we see that for any $t<0$,
  \[\begin{array}{l}
  \ds{\int_0^\infty \left<\psi_n(t-s), \phi(s) \right>_H ds = \int_{-\infty}^0 \left< \psi_n(s), \phi(t-s) \mathbbm{1}_{s<t} \right>_H ds}\\
  \vs
  \ds{ \to \int_{-\infty}^0 \left< \psi(s), \phi(t-s)\mathbbm{1}_{s<t} \right>_H ds = \int_0^\infty \left<\psi(t-s), \phi(s) \right>_H ds}
  \end{array}\]
  and therefore,
  \[\psi_{n}(t-\cdot) \to \psi(t-\cdot) \text{ weakly in } L^2((0,+\infty);H). \]
  We define
  \[\Psi(t) = L \psi(t-\cdot).\]
  By Lemma \ref{lem:weak->strong}, for any fixed $t\leq 0$, $\Psi_n(t) \to \Psi(t)$.
  Then, to finish the proof, we must show that (possibly for a further subsequence) this convergence is uniform on bounded intervals of $t$.
  By a generalization of the Arzela-Ascoli theorem, it remains to prove that $\{\Psi_{n}\}$ is a equicontinuous set.

  $\{\Psi_n\}$ is equicontinuous if and only if for any sequences $-\infty< s_n<t_n \leq0$ such that $t_n-s_n \to0$,
  \[\lim_{n \to +\infty} |\Psi_n(t_n) -\Psi_n(s_n)|_E = 0.\]

  Fix sequences $-\infty< s_n \leq t_n \leq 0$. Then
  \[\begin{array}{l}
  \ds{\Psi_n(t_n) - \Psi_n(s_n) = L \psi_n(t_n - \cdot) - L \psi_n(s_n -\cdot)}\\
    \ds{=\int_0^\infty S(r) Q \psi_n(t_n - r) dr - \int_0^\infty S(r) Q \psi_n(s_n -r) dr}\\
  \ds{= \int_0^{t_n-s_n} S(r) Q \psi_n(t_n -r) dr + (S(t_n - s_n) - I) \int_0^\infty S(r) Q \psi_n(s_n-r) dr}\\
  \ds{= L (\psi_n(t_n -\cdot) \mathbbm{1}_{[0,t_n-s_n]}(\cdot)) + (S(t_n - s_n) - I) L \psi_n(s_n-\cdot).}
  \end{array}\]

  Because $L$ is a compact operator, every subsequence of $\{L \psi_n(s_n - \cdot)\}$ has a convergent subsequence, and therefore
  \[ \lim_{n \to +\infty} |(S(t_n - s_n) - I) L \psi_n(s_n -\cdot)|_E =0.\]
  Next, we observe that as $|t_n-s_n| \to 0$
  \[\psi_n(t_n -\cdot) \mathbbm{1}_{[0,t_n-s_n]} \to 0 \text{ weakly in } L^2((0,+\infty);H).\]
  To see this, take any test function $\phi \in L^2((0,+\infty);H)$. By the Cauchy-Schwarz inequality
  \[\left|\int_0^{t_n - s_n} \left< \psi_n(t_n -r), \phi(r) \right>_H \right| \leq |\psi_n|_{L^2((0,+\infty);H)} |\phi|_{L^2((0,t_n-s_n);H)} \to 0.\]
  Therefore, because $L$ is a continuous operator from the weak topology on $L^2((0,+\infty);H)$ to $E$, it follows that
  \[\lim_{n \to +\infty} |L (\psi_n(t_n - \cdot) \mathbbm{1}_{[0,t_n -s_n]})|_E = 0.\]

  Therefore, the family $\{\Psi_n\}$ is equicontinuous, and for each $t<0$ $\Psi_n(t) \to \Psi(t)$. By the Arzela-Ascoli theorem, it follows that there exists a subsequence, also relabeled as $\Psi_n$, such that for any $T>0$
  \[\lim_{n \to +\infty}|\Psi_n - \Psi|_{C([-T,0];E)} = 0.\]
\end{proof}

\section{Semilinear equation, additive noise} \label{sec:additive}
In this section we extend the results from the previous section to the semilinear equation
\begin{equation}
  \left\{
  \begin{array}{l}
    \ds{dX^\e_x(t) = (AX^\e_x(t) + F(X^\e_x(t))) dt + \sqrt{\e}Q dw(s)}\\
    \ds{X^\e_x(0) = x}
  \end{array}
  \right.
\end{equation}
in its mild formulation
\begin{equation}
  X^\e_x(t) = S(t)x + \int_0^t S(t-s)F(X^\e_x(s)) ds + \sqrt{\e} \int_0^t S(t-s) Q dw(s).
\end{equation}
In what follows, we will assume that the following conditions are satisfied.
\begin{Hypothesis} \label{H-F-properties}
\begin{enumerate}
  \item[\textit{i.}] $A + F$ is strongly dissipative. That is, there exists $\lambda>0$ such that for any $x,y \in D(A)$, there exists $x^\star \in \partial|x-y|_E$ such that
  \begin{equation} \label{eq:dissipative-assum}
   \left<A(x-y) + F(x) - F(y), x^\star \right>_{E,E^\star} \leq -\lambda |x-y|_E,
  \end{equation}
  where the subdifferential $\partial|x-y|_E$ is defined as
  \begin{equation} \label{eq:subdifferential}
    \partial|x-y|_E = \{x^\star \in E^\star: |x^\star|_{E^\star} = 1, \left<x-y, x^\star\right>_{E,E^\star} = |x-y|_E\}
  \end{equation}

  \item[\textit{ii.}] $F$ is uniformly continuous on bounded subsets of $E$.

  \item[\textit{iii.}] $F(0) = 0$.
\end{enumerate}
\end{Hypothesis}

Dissipativity guarantees that the unperturbed system $X^0_x$ has a unique global solution (see for example \cite[Appendix D]{DaP-Z}) and the fact that $F(0) = 0$ implies that
\[\lim_{t \to +\infty} |X^0_x(t)|_E = 0.\]

By the contraction principle, we can extend the LDP for the stochastic convolution to an LDP for $X^\e_x$. The rate function for $X^\e_x$ will be based on the deterministic control system associated with this equation. Namely for $\psi \in L^2([0,T];H)$ we consider the problem
\begin{equation} \label{control-eq}
  \left\{
  \begin{array}{l}
    \ds{\frac{dX^\psi_x}{dt}(t) =AX^\psi_x(t) + F(X^\psi_x(t)) + Q \psi(t),}\\
    \vs
    \ds{X^\psi_x(0) = x,}
  \end{array}
  \right.
\end{equation}
which can be rewritten in mild form as
\begin{equation} \label{control-eq-mild}
 X^\psi_x(t) = S(t) x + \int_0^t S(t-s) F(X^\psi_x(s)) ds + \int_0^t S(t-s) Q \psi(s) ds.
\end{equation}

For any $T>0$, we define the rate functions
\begin{equation}
  I_{0,T}( \varphi) = \inf \left\{ \frac{1}{2} |\psi|_{L^2([0,T];H)}^2 : \varphi = X^\psi_{x} \right\}
\end{equation}
with the standard convention that $\inf \emptyset = +\infty$.
We also define the corresponding level sets, $K^x_{0,T}(r)$, by
\begin{equation} \label{K_0T-def} K^x_{0,T}(r) = \{\varphi \in C([0,T];E): \varphi(0) = x,  I_{0,T}(\varphi) \leq r  \}.\end{equation}

We now prove that $X^\e_x$ satisfies a large deviation principe which is uniform with respect to the initial condition on bounded subsets of $E$.
\begin{Theorem}[Uniform Large deviations principle] \label{thm:LDP}
  For any $\psi \in L^2([0,T];H)$, $R>0$ and $\delta>0$,
  \begin{equation} \label{LDP-lower}
    \liminf_{\e \to 0} \e \log \left( \inf_{|x|_E \leq R} \Pro \left( \left| X^\e_x - X^\psi_x \right|_{C([0,T];E)} <\delta \right) \right) \geq - \frac{1}{2} |\psi|_{L^2([0,T];H)}^2,
  \end{equation}
  and for any $r>0$, $R>0$, and $\delta>0$,
  \begin{equation} \label{LDP-upper}
    \limsup_{\e \to 0} \e \log \left(\sup_{|x|_E \leq R} \Pro \left( \textnormal{dist}_{C([0,T];E)}(X^\e_x, K^x_{0,T}(r)) > \delta \right) \right) \leq -r.
  \end{equation}
\end{Theorem}

Before proving the above theorem, we define the mapping $\alpha:  C([0,T];E) \to C([0,T];E)$ where
$\alpha(\Phi)(t)$ is the unique solution to
\begin{equation} \label{eq:alpha-def}
  \alpha(\Phi)(t) = \int_0^t S(t-s) F(\alpha(\Phi)(s)) ds + \Phi(t).
\end{equation}
In this way, $\alpha(S(\cdot)x + \sqrt{\e}Z(\cdot)) = X^\e_x$ and $\alpha(S(\cdot) x + Z^\psi(\cdot)) = X^\psi_x$.

We can show that $\alpha$ is well defined by proving that the mapping $K: C([0,T];E) \to C([0,T];E)$ given by
\[K(u)(t) = \int_0^t S(t-s) F(u(s)) ds + \Phi(t)\]
is a contraction for small $T$ and then using a bootstrap argument. To do this we need the following a priori bound.
\begin{Lemma} \label{lem:alpha-bound}
  There exists a continuous increasing function $\kappa(r)$ such that for any $T>0$,
  \begin{equation} \label{eq:alpha-bound}
  |\alpha(\Phi)|_{C([0,T];E)} \leq  \kappa( |\Phi|_{C([0,T];E)}).
  \end{equation}
\end{Lemma}
\begin{proof}
  Let $\varphi = \alpha(\Phi)$. Then
  \[\varphi(t) - \Phi(t) = \int_0^t S(t-s) F(\varphi(s)) ds.\]
  This is weakly differentiable and
  \[\frac{d}{dt}[\varphi(t) - \Phi(t)] = A[\varphi(t) - \Phi(t)] + F(\varphi(t)) = A[\varphi(t) - \Phi(t)] + F(\varphi(t)) - F(\Phi(t)) + F(\Phi(t)).\]
  By Hypothesis \ref{H-F-properties}, there exists $x^\star(t) \in \partial|\varphi(t) - \Phi(t)|_E$ such that
  \[\frac{d^-}{dt} |\varphi(t) - \Phi(t)|_E = \left<\frac{d}{dt}[\varphi(t) - \Phi(t)], x^\star(t) \right>_{E,E^\star}  \leq -\lambda|\varphi(t) - \Phi(t)|_E + |F(\Phi(t))|_E.\]
  Therefore,
  \begin{equation} \label{eq:apriori-bound}
  |\varphi(t) - \Phi(t)|_E \leq |\varphi(0) - \Phi(0)| e^{-\lambda t} + \int_0^t e^{-\lambda (t-s)} |F(\Phi(s))|_E ds \leq   |\Phi(0)| + \frac{1}{\lambda} \sup_{s \leq t} |F(\Phi(s))|_E.
  \end{equation}
  Then we can conclude that
  \[|\varphi|_{C([0,T];E)} \leq |\varphi - \Phi|_{C([0,T];E)} + |\Phi|_{C([0,T];E)} \leq 2|\Phi|_{C([0,T];E)} + \frac{1}{\lambda} \sup_{t \leq T} |F(\Phi(t))|_E.\]
  We can set
  \[\kappa(r) = 2r + \frac{1}{\lambda}\sup_{|y|_E\leq r} |F(y)|_E.\]
\end{proof}

Next we show that $\Phi \mapsto \alpha(\Phi)$ is uniformly continuous on bounded subsets of $\Phi$.
\begin{Lemma} \label{lem:continuity}
For any $\delta>0$ and any $\rho>0$, there exists $\delta_\rho>0$ such that if $ |\Phi_0|_{C([0,T];E)} \leq \rho$ and $|\Phi_0 - \Phi|_{C([0,T];E)} < \delta_\rho$, then
\[ \left|\alpha(\Phi)- \alpha(\Phi_0) \right|_{C([0,T];E)} < \delta.\]
\end{Lemma}
\begin{proof}
  Let $\varphi = \alpha(\Phi)$, and $\varphi_0 = \alpha(\Phi_0)$.
  By subtracting these we see that
  \[\varphi_0(t) - \varphi(t) -\Phi_0(t) + \Phi(t) = \int_0^t S(t-s) [F(\varphi_0(s)) - F(\varphi(s))]ds.\]
  This is differentiable in a weak sense and
  \[\frac{d}{dt}[\varphi_0(t) - \varphi(t) -\Phi_0(t) + \Phi(t)] = A [\varphi_0(t) - \varphi(t) -\Phi_0(t) + \Phi(t)] + F(\varphi_0(t)) - F(\varphi(t)).\]
  We rewrite this as
  \[\begin{array}{l}
  \ds{\frac{d}{dt}[\varphi_0(t) - \varphi(t) -\Phi_0(t) + \Phi(t)] =}\\
   \vs
   \ds{A [\varphi_0(t) - \varphi(t) -\Phi_0(t) + \Phi(t)] +[ F(\varphi_0(t) - \Phi_0(t) + \Phi(t)) - F(\varphi(t))]}\\
    \vs
    \ds{+ [F(\varphi_0(t)) - F(\varphi_0(t) - \Phi_0(t) + \Phi(t))].}
  \end{array}\]

  By the same arguments that we used in the proof of the previous lemma,
  \begin{equation} \label{eq:varphi-continuity-bound}
  \begin{array}{l}
  \ds{|\varphi_0(t) - \varphi(t) -\Phi_0(t) + \Phi(t)|_E }\\
  \vs
  \ds{  \leq |\Phi(0) - \Phi_0(0)|_E +   \int_0^t e^{-\lambda(t- s)} |F(\varphi_0(s)) - F(\varphi_0(s) - \Phi_0(s) + \Phi(s))|_E ds}\\
  \ds{\leq |\Phi(0) - \Phi_0(0)|_E +  \frac{1}{\lambda} \sup_{s\leq t}|F(\varphi_0(t)) - F(\varphi_0(t) - \Phi_0(t) + \Phi(t))|_E.}
  \end{array}
  \end{equation}

  By Hypothesis \ref{H-F-properties}, $F$ is uniformly continuous on bounded subsets of $E$ and by \eqref{eq:alpha-bound}, $|\varphi_0|_{C([0,T];E)}$ is bounded uniformly for $|\Phi_0|_{C([0,T];E)} \leq \rho$. Therefore we can find $\delta_\rho>0$ so that \[|\Phi_0 - \Phi|_{C([0,T];E)} < \delta_\rho \Longrightarrow |\varphi_0 - \varphi -\Phi_0 + \Phi|_{C([0,T];E)} < \frac{\delta}{2}. \]

  We conclude the proof with the observation that
  \[|\varphi - \varphi_0|_{C([0,T];E)} \leq |\varphi - \varphi_0 -\Phi + \Phi_0|_{C([0,T];E)} + |\Phi_0 - \Phi|_{C([0,T];E)}\]
  so our result follows by possibly decreasing $\delta_\rho$.
\end{proof}

\begin{Remark}
  The uniqueness of the mild solution of equation \eqref{eq:mild-solution} follows from the previous lemma.
\end{Remark}

\begin{proof}[Proof of Theorem \ref{thm:LDP} (Uniform large deviations principle)]

  First, by Theorem \ref{thm:L-compact}, the operator $L: L^2((0,+\infty];H) \to E$ is compact and therefore bounded. This means that there exists $c>0$ such that
  \[\left| \int_0^\infty S(t) Q\psi(s) ds \right|_E \leq c |\psi|_{L^2([0,+\infty);H)}.\] By a time change, it follows that
  \[\left|Z^\psi(t)\right|_E = \left|\int_0^t S(t-s) Q \psi(s) ds\right|_E \leq c |\psi|_{L^2([0,t];H)}.\]
  Let $R>0$, $r>0$ and set $\rho = R + c\sqrt{2r}$. For any $\delta>0$ let $\delta_\rho>0$ be from Lemma \ref{lem:continuity}. Then for any $\frac{1}{2}|\psi|_{L^2((0,T);H)}^2 \leq r$,
  \[ \inf_{|x|_E \leq R} \Pro \left(|X^\e_x - X^\psi_x |_{C([0,T];E)} < \delta \right) \geq \Pro \left( \left|Z^\psi -  \sqrt{\e}Z \right|_{C([0,T];E)} < \delta_\rho \right).\]
  Consequently, \eqref{LDP-lower} follows from \eqref{eq:stoch-conv-LDP-lower}.
  By almost the same argument,
  \[\sup_{|x|_E \leq R} \Pro \left( \dist_{C([0,T];E)} (X^\e_x, K_{0,T}^x(r) ) > \delta \right) \leq \Pro \left( \dist_{C([0,T];E)} (\sqrt{\e} Z, \tilde{K}_{0,T}(r)) > \delta_\rho \right) \]
  and \eqref{LDP-upper} follows from \eqref{eq:stoch-conv-LDP-upper}.
\end{proof}

\begin{Remark}
  In the case where $F$ is globally Lipschitz continuous, the large deviation principle is uniform with respect to all $x \in E$. This can be proven with a straightforward application of Gr\"onwall's inequality.
\end{Remark}
We now establish that the solution $X^\psi_x$ depends continuously on the initial condition $x$.
\begin{Theorem} \label{thm:init-value-continuity}
  For any $x_1, x_2 \in E$, and $\psi \in L^2([0,T];H)$,
  \begin{equation}\label{eq:init-value-continuity}
    |X^\psi_{x_1} - X^\psi_{x_2}|_{C([0,T];E)} \leq |x_1 - x_2|_E.
  \end{equation}
\end{Theorem}

\begin{proof}
  We proceed as in Lemma \ref{lem:alpha-bound}. Let
  \[\varphi(t):=X^\psi_{x_1}(t) - X^\psi_{x_2}(t) = S(t)(x_1-x_2) + \int_0^t S(t-s)(F(X^\psi_{x_1}(s)) - F(X^\psi_{x_2}(s))) ds.\]
  The weak derivative of the above expression is
  \[\frac{d}{dt}\varphi(t) = A \varphi(t) + F(X^\psi_{x_1}(t)) - F(X^\psi_{x_2}(t)).\]
  Therefore, by Hypothesis \ref{H-F-properties} there exists $x^\star(t) \in \partial|\varphi(t)|_E$ such that
  \[\frac{d^-}{dt}|\varphi(t)|_E \leq \left<A \varphi(t) + F(X^\psi_{x_1}(t)) - F(X^\psi_{x_2}(t)), x^\star(t)\right>_{E,E^\star} \leq 0.\]
  The result follows because $|\varphi(0)|_E = |x-y|_E$.
\end{proof}

Now, we want to extend the domain of the rate functions to include trajectories on the negative half-line. Because \eqref{eq} is time homogeneous, we use translation to define, for any $t<T \in \mathbb{R}$,
\begin{equation} \label{I_t,T-def}
  I_{t,T}(\varphi) = I_{0,T - t}(\varphi( \cdot + t)).
\end{equation}
We then define
\begin{equation} \label{I_-infty-def}
  I_{-\infty,T}(\varphi) = \sup_{t < T} I_{t,T}(\varphi).
\end{equation}
Now, as mentioned in the introduction, we define the quasipotential to be, for any $N \subset E$
\begin{equation} \label{V-def}
  V(N) = \inf \left\{ I_{-\infty,0}(\varphi) : \lim_{t \to -\infty} |\varphi(t)|_E = 0, \ \varphi(0) = N \right\}
\end{equation}
where, once again, we use the standard convention that $\inf \emptyset = +\infty$.
One can think of this as the minimal amount of energy required to reach the set $N$ starting from the point $0$ in an infinite amount of time.

If $\varphi \in C((-\infty,0);E)$ and $I_{-\infty,0}(\varphi)<+\infty,$ then by \eqref{I_t,T-def} and \eqref{I_-infty-def} $\varphi$ is the weak solution to a control problem for some $\psi \in L^2((-\infty,0);H)$. That is, for any $-T<t<0$,
\[\varphi(t) = S(t+T)\varphi(-T) + \int_{-T}^t S(t-s) F(\varphi(s))ds  + \int_{-T}^t S(t-s) Q \psi(s) ds. \]

\begin{Theorem}[Compact level sets] \label{thm:compact-level-sets}
  Under Hypotheses \ref{H:linear}, \ref{H:QandZ}, and \ref{H-F-properties}, for any $r \in \mathbb{R}$, the set
  \begin{equation} \label{K_-infty-def}
    K_{-\infty}(r) = \{\varphi \in C((-\infty,0);E) : \lim_{t \to -\infty}|\varphi(t)|_E=0, \ I_{-\infty,0}(\varphi)\leq r\}
  \end{equation}
   is compact in the topology of uniform convergence on bounded intervals.
\end{Theorem}
\begin{proof}
  Consider a sequence $\{\varphi_n\} \subset K_{-\infty}\left(r \right)$. We will show that there exists a subsequence that converges in $C((-\infty,0];E)$. First, we observe that there must exist $\{\psi_n\} \subset L^2((0,+\infty);H)$ with $\frac{1}{2}|\psi_n|_{L^2((-\infty,0);H)} \leq r +\frac{1}{n}$ such that for any $-\infty < -T<t<0$,
  \[\varphi_n(t) = S(t+T)\varphi(-T)+ \int_{-T}^t S(t-s) F(\varphi_n(s))ds + \int_{-T}^t S(t-s)Q \psi_n(s) ds.\]
  By Lemma \ref{lem:alpha-bound}, for any $-T\leq t \leq 0 $,
  \begin{equation} \label{eq:varphi-n-a-priori}
    |\varphi_n(t)|_E \leq  \kappa\left( |\varphi_n(-T)|_E + \left| \Psi^T_n\right|_{C([-T,t];E)}  \right)
  \end{equation}
  where
  \[\Psi^T_n(t) = \int_{-T}^t S(t-s)Q \psi_n(s) ds. \]
  We also set
  \[\Psi_n(t) = \int_{-\infty}^t S(t-s)Q \psi_n(s) ds.\]
  Note that by Theorem \ref{thm:L-compact}, there exists $c>0$ such that
  \[\left| \Psi^T_n(t) \right|_E = \left|L \psi_n(t-\cdot) \mathbbm{1}_{\{\cdot>-T\}}\right|_E \leq c |\psi_n|_{L^2((-\infty,0);H)}.\]
  If we let $T \to +\infty$ in \eqref{eq:varphi-n-a-priori}, then because $\lim_{T \to +\infty}\varphi_n(-T) = 0 $,
  we see that
  \begin{equation} \label{eq:apriori-bound-2}
  \sup_{n \in \nat} \sup_{t \leq0}|\varphi_n(t)|_E <+\infty.
  \end{equation}

  Now we show that there exists a subsequence of $\{\varphi_n\}$ that is a Cauchy sequence in $C([t,0];E)$ for any $t<0$.
  By Lemma \ref{thm:linear-K-compact}, there exists a subsequence of $\{\Psi_n\}$, which we relabel as $\Psi_n$ which converges uniformly on $C([t,0];E)$ for any $t<0$.
  By the arguments of Lemma \ref{lem:continuity} and especially \eqref{eq:varphi-continuity-bound}, we see that for any $T>0$, $-T<t<0$,
  \[\begin{array}{l}
  \ds{|\varphi_n(t) - \varphi_m(t)|_E \leq |\Psi^T_n(t) - \Psi^T_m(t)|_E + |\varphi_n(-T) -\varphi_m(-T) - \Psi^T_n(-T) + \Psi_m(-T)|_E}\\
   \vs
   \ds{+ \int_{-T}^t e^{-\lambda(t- s)} |F(\varphi_n(s) -\Psi^T_n(s) + \Psi^T_m(s)) - F(\varphi_n(s))|_E ds.}
   \end{array}\]
  Notice that by definition, $\Psi^T_n(-T)=\Psi^T_m(-T)=0$. By letting $T \to +\infty$, we see that
  \[|\varphi_n(t) - \varphi_m(t)|_E \leq |\Psi_n(t) - \Psi_m(t)|_{C([0,t];E)} + e^{-\lambda t} \int_{-\infty}^t e^{\lambda s} |F(\varphi_n(s) -\Psi_n(s) + \Psi_m(s)) - F(\varphi_n(s))|_E ds.\]
  Because $\{\Psi_{n}\}$ is a Cauchy sequence on compact subsets, the dominated convergence on the integral term implies that
  $\{\varphi_n\}$ is Cauchy on $C([t,0];E)$ for any $t<0$. Therefore, by completeness, there is a limit $\varphi$ such that $\varphi_n \to \varphi$ uniformly on compact sets.
   We can use the dominated convergence because the  $\Psi_n,$ and $\varphi_n$ are uniformly bounded and $F$ is continuous on bounded subsets.

   It remains to show that $\varphi \in K_{-\infty}(r)$. By Alaoglu's theorem, we can take a subsequence of the $\{\psi_n\}$ that converges weakly to a limit $\psi$ with $\frac{1}{2}|\psi|_{L^2((-\infty,0);H)} \leq r$. For any $T>0$, $-T<t\leq 0$
   \[\varphi_n(t) = S(t+T)\varphi_n(-T) + \int_{-T}^t S(t-s) F(\varphi_n(s)) ds + \int_{-T}^t S(t-s)Q \psi_n(s) ds. \]
   By Lemma \ref{lem:weak->strong} and Theorem \ref{thm:linear-K-compact},
   \[\lim_{n \to +\infty} \int_{-T}^t S(t-s) Q \psi_n(s )ds = \int_{-T}^t S(t-s) \psi(s) ds\]
   and therefore, by Lemma \ref{lem:continuity}, because $\varphi_n(-T) \to \varphi(-T)$,  it follows that $\varphi$ solves
   \[\varphi(t) = S(t+T) \varphi(-T) + \int_{-T}^t S(t-s) F( \varphi(s)) ds + \int_{-T}^t S(t-s) Q \psi(s) ds.\]

   The last thing to check is that $\lim_{t \to -\infty} |\varphi(t)|_E = 0$. By \eqref{eq:apriori-bound} and \eqref{eq:apriori-bound-2}, we see that for any $-T<t\leq 0$
   \[|\varphi(t)|E \leq c e^{-\lambda(t+T)} + \frac{1}{\lambda} \sup_{-T \leq s \leq t} |F(\Psi^T(s))|_E, \]
   where
   \[\Psi^T(s) = \int_{-T}^s S(t-s) Q \psi(s) ds.\]
   By letting $T \to +\infty$, we see that
   \[|\varphi(t)|_E \leq \frac{1}{\lambda} \sup_{s \leq t} |F(\Psi(s))|_E.\]
   Then we observe that
   \[|\Psi(t)|_E = |L\psi(t-\cdot)|_E \leq \|L\|_{\mathcal{L}(L^2((-\infty,0);H),E)} |\psi|_{L^2((-\infty,t);H)}.\]
   As $t \to \infty$, $F(\Psi(t))$ goes to $0$ because $F$ is continuous with $F(0)=0$. We conclude that
   \[\lim_{t \to -\infty} |\varphi(t)|_E = 0.\]
\end{proof}


There are some interesting immediate corollaries of Theorem \ref{thm:compact-level-sets}.
\begin{Corollary} \label{cor:I-is-minimum}
  If $I_{-\infty,0}(\varphi) < +\infty$, then there exists $\psi \in L^2((-\infty,0);H)$ such that for any $-T<t< +\infty$,
  \[\varphi(t)= S(T+t) \varphi(-T) + \int_{-T}^t S(t-s) F(\varphi(s))ds + \int_{-T}^t S(t-s) Q \psi(s) ds\]
  and
  \[\frac{1}{2}|\psi|_{L^2((-\infty,0);H)}^2 = I_{-\infty,0}(\varphi).\]
\end{Corollary}

\begin{Corollary} \label{cor:V-is-minimum}
  If $V(N) <+\infty$, then there exists $\varphi \in C((-\infty,0];E)$ such that $\varphi(0) \in N$, $\lim_{t \to -\infty} \varphi(t) = 0$ and
  \[V(N) = I_{-\infty,0}(\varphi).\]
\end{Corollary}

\section{Multiplicative noise} \label{sec:multiplicative}
In this section we consider a special case of multiplicative noise and we show that Theorem \ref{thm:compact-level-sets} still holds.

We consider the following stochastic equation
\begin{equation} \label{eq:mult-SDE}
  \left\{
  \begin{array}{l}
   \ds{dX^\e_x(t) = (AX^\e_x(t) + F(X^\e_x(t)))dt + \sqrt{\e}Q B(X^\e_x(t))dw(t),}\\
   \vs
   \ds{X^\e_x(0) = x.}
  \end{array}
  \right.
\end{equation}
and its associated mild formulation
\begin{equation} \label{eq:mult-mild}
  X^\e_x(t) = S(t)x + \int_0^t S(t-s)F(X^\e_x(s)) ds + \sqrt{\e} \int_0^t S(t-s) Q B(X^\e_x(s)) dw(s).
\end{equation}

We also study the associated deterministic control problem
\begin{equation} \label{eq:mult-control}
  \left\{
  \begin{array}{l}
   \ds{\frac{d}{dt}X^\psi_x(t) = AX^\psi_x(t) + F(X^\psi_x(t)) + Q B(X^\psi_x(t))\psi(t),}\\
   \vs
   \ds{X^\e_x(0) = x,}
  \end{array}
  \right.
\end{equation}
and its mild formulation
\begin{equation} \label{eq:mult-contr-mild}
  X^\psi_x(t) = S(t)x + \int_0^t S(t-s)F(X^\psi_x(s))ds + \int_0^t S(t-s)QB(X^\psi_x(s)) \psi(s) ds.
\end{equation}
For the operator $Q$, we still assume that Hypothesis \ref{H:QandZ} holds. Moreover we assume the following condition on $B$.
\begin{Hypothesis}\label{H:B-multiplicative}
  $B: E \to \mathcal{L}(H)$ and there exists $\kappa>0$ such that for $x, y \in E$, $h \in H$,
  \begin{equation} \label{eq:B-hyp}
   \begin{array}{l}
    \ds{|(B(x) - B(y))h|_H \leq \kappa |x-y|_E |h|_H}\\
    \vs
    \ds{|B(x)h|_H \leq \kappa (1 + |x|_E) |h|_H.}
   \end{array}
  \end{equation}
\end{Hypothesis}

\begin{Example}
  Let $D \subset \mathbb{R}^d$ and let $H=L^2(D)$. Then if $\{e_k\}$ is a complete orthonormal basis of $H$ and $\{\beta_k\}$ is a sequence of independent one-dimensional Brownian motions, the formal sum
  \[dw(t,\xi) =  \sum_{k=1}^\infty e_k(\xi) d\beta_k(t) \]
  is a space-time white noise. Let $E = C(D)$ be the space of continuous functions on $D$.

  Let $b: \mathbb{R} \to \mathbb{R}$ be a Lipschitz continuous function. In particular we assume that there exists $\kappa>0$ such that for $r,s \in \mathbb{R}$
  \begin{equation} \label{eq:Lipschitz}
    |b(r) - b(s)| \leq \kappa |r-s|, \qquad |b(r)| \leq \kappa(1 + |r|).
  \end{equation}
  and define $B: E \to \mathcal{L}(H)$ to be the multiplication operator for $x \in E$ and $h \in H$
  \[(B(x)h)(\xi) = b(x(\xi))h(\xi).\]
  By \eqref{eq:Lipschitz}, for $x,y \in E$, $h \in H$,
  \[
  |B(x)h|_H^2 = \int\limits_D |b(x(\xi))h(\xi)|^2 d\xi \leq \kappa^2 (1 + |x|_E)^2 |h|_H^2,
  \]
  \[|(B(x) - B(y))h|_H^2 = \int\limits_D |(b(x(\xi))-b(y(\xi)))h(\xi)|^2 d\xi \leq \kappa^2 |x-y|_E^2 |h|_H^2.  \]
\end{Example}

Recall that in Hypothesis \ref{H-F-properties} we assumed that $F$ is uniformly continuous on bounded regions of $F$ and that $A+F$ is strongly dissipative. If $B: E \to \mathcal{L}(H)$ were bounded, then Hypothesis \ref{H-F-properties} would be appropriate for this section. Unfortunately, because we assumed that $B$ has linear growth we need to strengthen Hypothesis \ref{H-F-properties} to guarantee that the mild solutions \eqref{eq:mild-control} are well defined. We need to assume that the growth rate of $F$ is compensated by its dissipativty.

\begin{Hypothesis} \label{H:F-mult}
 The function $F$ has polynomial growth and polynomial dissipativity. There exists $\kappa>0$, $\lambda>0$ and $m \in \nat$ such that
 \begin{enumerate}
   \item[\textit{i}.] For any $x \in E$, $|F(x)|_E \leq \kappa( 1+  |x|_E^m)$.
   \item[\textit{ii}.] For any $x, h \in D(A)$, there exists $h^\star \in \partial|h|_E$ such that
    \begin{equation} \label{eq:polynomial-dissipative}
      \left< F(x+h) - F(x),h^\star\right>_{E,E^\star} \leq -\lambda |h|^m + \kappa (1 + |x|^m_E).
    \end{equation}
 \end{enumerate}
\end{Hypothesis}
\begin{Example}
 Let $E = C(D)$ and let $F$ be the Nemytskii operator
 \[F(x)(\xi) = - (x(\xi))^3.\]
 Such an $F$ satisfies Hypothesis \ref{H:F-mult}. The polynomial growth is clear. To see the dissipativity of this operator, observe that for any $x,h \in E$, $\xi \in D$,
 \[\begin{array}{l}
 \ds{F(x+h)(\xi) - F(x)(\xi)= -(x(\xi)+h(\xi))^3 + (x(\xi))^3 = - h^3(\xi) + 3x^2(\xi)h(\xi) + 3 h^2(\xi) x(\xi)}\\
 \vs
 \ds{\leq -\frac{1}{2}h^3(\xi) + \kappa(1 + |x(\xi)|^3)  }
 \end{array}\]
 where the last line follows from Young's inequality. Next, we recall that all $h^\star \in \partial|h|_E$, are signed measures with total variation $\|h^\star\|_{TV}=1$, which are positive on the set
 \[\{\xi \in D: h(\xi) = |h|_E\},\]
 negative on the set
 \[\{\xi \in D: h(\xi) = -|h|_E\}\]
 and $0$ everywhere else. Therefore,
 \[\left<h^3, h^\star \right>_{E,E^\star} = |h|_E^3.\]
 We conclude that for any $h^\star \in \partial|h|_E$,
 \[\left<F(x+h) - F(x), h^\star \right> \leq -\frac{1}{2}|h|^3_E + \kappa(1 + |x|_E^3).\]
\end{Example}

\begin{Lemma}[A priori bounds] \label{lem:apriori-bounds}
  For any $R>0$,
  \begin{equation}
    \sup \left\{|X^\psi_x(t)|_E : t\geq 0, |x|_E \leq R, |\psi|_{L^2((0,+\infty);H)} \leq R \right\} < +\infty.
  \end{equation}
\end{Lemma}
\begin{proof}
  Let
  \begin{equation}
    Z^\psi_x(t) = \int_0^t S(t-s)QB(X^\psi_x(s))\psi(s)ds.
  \end{equation}
  Then
  \begin{equation}
    X^\psi_x(t) - Z^\psi_x(t) = S(t)x + \int_0^t S(t-s) F(X^\psi_x(s)) ds.
  \end{equation}
  This expression is weakly differentiable and
  \[\frac{d}{dt}\left( X^\psi_x(t) - Z^\psi_x(t) \right) = A (X^\psi_x(t) - Z^\psi_x(t)) + F(X^\psi_x(t)).\]
  It follows that there exists $x^\star(t) \in \partial|X^\psi_x(t) - Z^\psi_x(t)|_E$ such that
  \begin{equation}
   \begin{array}{l}
    \ds{\frac{d^-}{dt} |X^\psi_x(t) - Z^\psi_x(t)|_E \leq \left<A(X^\psi_x(t) - Z^\psi_x(t)) + F(X^\psi_x(t)) - F(Z^\psi_x(t)), x^\star(t) \right>_{E, E^\star}}\\
    \vs
    \ds{ \hspace{3cm}+ |F(Z^\psi_x(t))|_E.}
   \end{array}
  \end{equation}
  By Hypothesis \ref{H:F-mult},
  \begin{equation}
    \frac{d^-}{dt} |X^\psi_x(t) - Z^\psi_x(t)| \leq -\lambda |X^\psi_x(t) - Z^\psi_x(t)|_E^m + \kappa (1 + |Z^\psi_x(t)|_E^m).
  \end{equation}
  If any nonnegative real-valued function has the property that
  \begin{equation} \label{eq:ODE-mth-power}
    \frac{du}{dt}(t) \leq -\lambda u^m(t) + \varphi^m(t)
  \end{equation}
  it follows that
  \[u(t) \leq u(0) + 2\lambda^{-\frac{1}{m}} \varphi(t).\]
  The above fact is true becuase if $u(t)$ were ever greater than $2 \lambda^{-\frac{1}{m}} \varphi(t)$, then according to \eqref{eq:ODE-mth-power}, $u'(t)$ would be negative. Therefore, the only way this can happen is if the initial condition is large.
  Therefore,
  \begin{equation}
    |X^\psi_x(t) - Z^\psi_x(t)|_E \leq |x|_E + c (1 + |Z^\psi_x(t)|_E)
  \end{equation}
  and
  \begin{equation}
    |X^\psi_x(t)|_E \leq |x|_E + c( 1 + |Z^\psi_x(t)|_E).
  \end{equation}
  Meanwhile, by Hypothesis \ref{H:B-multiplicative},
  \begin{equation}
   \begin{array}{l}
    \ds{|Z^\psi_x(t)|_E = \left|\int_0^t S(t-s)Q B(X^\psi_x(s)) ds \right|_E \leq \|L\|_{\mathcal{L}(L^2((0,t);H);E)} \left|B(X^\psi_x(\cdot))\psi \right|_{L^2((0,t);H)}}\\
    \vs
    \ds{\leq c \sqrt{ \int_0^t (1 +|X^\psi_x(s)|_E^2) |\psi(s)|_H^2 }.}
   \end{array}
  \end{equation}
  Therefore,
  \begin{equation}
    |X^\psi_x(t)|_E^2 \leq c \left( |x|_E^2 + 1 + |\psi|_{L^2((0,t);H)}^2 + \int_0^t |X^\psi_x(s)|_E^2 |\psi(s)|_H^2 ds \right)
  \end{equation}
  and by Gr\"onwall's inequality we can conclude that
  \begin{equation} \label{eq:apriori-concl}
    |X^\psi_x(t)|_E^2 \leq c \left(|x|_E^2 + 1 + |\psi|_{L^2((0,t);H)}^2 \right) e^{c |\psi|_{L^2((0,t);H)}^2}.
  \end{equation}
\end{proof}
The previous lemma allows us to extend the compactness results from the additive noise case to the results for the multiplicative noise case.
\begin{Definition}
  For any $T>0$ define the rate functions
  \begin{equation}
    I_{0,T}(\varphi) = \inf\{ \frac{1}{2}|\psi|_{L^2([0,T];H)}^2: \varphi(t) = X^\psi_x(t)\}.
  \end{equation}
\end{Definition}
We claim that Theorem \ref{thm:LDP} holds in the multiplicative noise case also. Unlike in the additive noise case, one can not apply a contraction principle. One must use variational methods (see for example \cite[Section 6]{cerrok-LDP}).
\begin{Theorem} \label{thm:compact-mult}
  Under Hypotheses \ref{H:linear}, \ref{H:QandZ}, \ref{H:B-multiplicative}, and \ref{H:F-mult}, for any $r>0$, the set of trajectories
  \begin{equation}
    \mathcal{K}(r) := \left\{ \varphi \in C((-\infty,0];E) : \lim_{t \to -\infty} |\varphi(t)|_E = 0, \ I_{-\infty}(\varphi) \leq r \right\}
  \end{equation}
  is compact in the topology of uniform convergence on bounded intervals.
\end{Theorem}
\begin{proof}
  Let $\{\varphi_n\} \subset \mathcal{K}(r)$ be a sequence. Then there exists $|\psi_n|_{L^2((-\infty,0);H)} \leq  \sqrt{2( r + \frac{1}{n})}$ so that for any $-T< t <0$,
  \[\varphi_n(t) = S(t+T)\varphi_n(-T) + \int_{-T}^t S(t-s) F(\varphi_n(s)) ds + \int_{-T}^t S(t-s) Q B(\varphi_n(s)) \psi_n(s) ds.\]
  By \eqref{eq:apriori-concl},
  \[\sup_{-T \leq t \leq 0} |\varphi_n(t)|_E  \leq c \left( |\varphi_n(-T)|_E + 1 +|\psi|_{L^2((-T,0);H)}^2 \right)e^{c|\psi|_{L^2((-T,0);H)}^2}.\]
  If we let $-T \to -\infty$, we see that for any $n \in \nat$,
  \[\sup_{ t\leq 0} |\varphi_n(t)|_E \leq c(1+2(r+1))e^{1+2(r+1)}\]
  where we used the fact that $|\psi_n|_{L^2((-\infty,0);H)} \leq\sqrt{2( r+1)}$.
  Let $v_n(t) = B(\varphi_n(t))\psi_n(t)$.
  Then, by Hypothesis \ref{H:B-multiplicative},
  \[\sup_n|v_n|_{L^2((-\infty,0);H)} \leq \sup_n \kappa (1 + \sup_{t \leq 0} |\varphi_n(t)|_E)|\psi_n|_{L^2((-\infty,0);H)} < +\infty.\]
  Therefore, by Alaoglu's theorem there is a subsequence which we relabel $v_n$ such that $v_n$ converges weakly in $L^2((-\infty,0);H)$ to a limit $v$.
  Then, by Theorem \ref{thm:compact-level-sets} about compactness in the additive noise case, we see that
  $\varphi_n$ converges uniformly on bounded intervals to $\varphi$ solving
  \[\varphi(t) = S(t+T) \varphi(-T) + \int_{-T}^t S(t-s)F(\varphi(s))ds + \int_{-T}^t S(t-s)Q v(s).\]
  It remains to argue that $v(s) = B(\varphi(s))\psi(s)$ for some $\psi(s)$. In fact, again by Alaoglu's theorem, there is a subsequence of $\psi_n$ which converges weakly in $L^2((-\infty,0);H)$ to $\psi$. Then, because $\varphi_n \to \varphi$ uniformly on compact intervals, for any test function $\phi \in L^2((-\infty,0);H)$ with finite support,
  \[\begin{array}{l}
  \ds{\left|\int_{-\infty}^0 \left<B(\varphi_n(s))\psi_n(s) - B(\varphi(s))\psi(s), \phi(s)\right>_H ds  \right|}\\
  \vs
  \ds{\leq \left|\int_{-\infty}^0 \left< B(\varphi(s))(\psi_n(s) - \psi(s)),\phi(s) \right>_H ds \right|
  + \left|\int_{-\infty}^0 \left< (B(\varphi_n(s)) - B(\varphi(s)))\psi_n(s), \phi(s) \right>_H ds \right|}\\
  \vs
  \ds{\leq \left|\int_{-\infty}^0 \left<\psi_n(s) - \psi(s), B^\star(\varphi(s))\phi(s) \right>_H \right| + \left|\int_{-\infty}^0 \left<B(\varphi_n(s))- B(\varphi(s)) \psi_n(s), \phi(s) \right>_H \right|}\\
  \vs
  \ds{:=I^n_1 + I^n_2}
  \end{array}\]
  $I^n_1$ converges to zero because $\psi_n$ converge weakly to $\psi$ in $L^2((0,+\infty);H)$. $I^n_2$ converges to zero because $\varphi_n$ converges to $\varphi$ uniformly on compact intervals and $\phi$ has finite support. Therefore $B(\varphi_n(s))\psi_n(s) \to B(\varphi(s))\psi(s)$ weakly in $L^2((-\infty,0);H)$. Therefore $v(s) = B(\varphi(s))\psi(s)$ and  the limit $\varphi \in \mathcal{K}(r)$.
\end{proof}

\section{Exit time and exit place} \label{sec-exit-proofs}
Let $G \subset E$ be an open, connected, and bounded set, and let $0 \in G$. The goal of the next section is to characterize the exit time and exit place of the process $X^\e_x$ from the domain $G$.
\begin{Lemma}[Attraction to stable equilibrium] \label{lem:attraction}
  The unperturbed equation $X^0_x$ converges to $0$ uniformly for $x \in G$. That is
  \begin{equation} \label{eq:attraction}
  \lim_{t \to +\infty} \sup_{x \in G} |X^0_x|_E = 0.
  \end{equation}
  Also, for any $t>0$, $|X^0_x(t)|_E \leq |x|_E$.
\end{Lemma}

\begin{proof}
  This is a straightforward consequence of dissipativity. $t \mapsto X^0_x(t)$ is weakly differentiable and
  \[\frac{d}{dt}X^0_x(t) = A X^0_x(t) + F(X^0_x(t)). \]
  Therefore,
  \[\frac{d^-}{dt}|X^0_x(t)|_E \leq -\lambda |X^0_x(t)|_E\]
  and it follows that
  \[|X^0_x(t)|_E \leq e^{-\lambda t} |x|_E.\]
  Our results follow because $G$ is bounded.
\end{proof}

\begin{Hypothesis}[Boundary Regularity] \label{H4}
  The boundary of $G$ is regular enough so that
  \[ V(\partial G) = V(\bar{G}^c) <+\infty.\]
\end{Hypothesis}

\begin{Remark}
  First, we observe that it is always the case that
  \[V(\partial G) \leq V(\bar{G}^c).\]
  If $ \tilde{y} \in \bar{G}^c$ and $V(\tilde{y})< +\infty$, then by Corollary \ref{cor:V-is-minimum} there exists a trajectory $\varphi \in C((-\infty,0);E)$ such that
  \[\lim_{t \to -\infty} \varphi(t) = 0 \text{ and } \varphi(0) = \tilde{y}\]
  with
  \[I_{-\infty,0}(\varphi) =V(\tilde{y}) < +\infty.\]
  Because this trajectory is continuous and we assumed that $0 \in G$, and $\tilde{y} \in \bar{G}^c$, there must be some $t<0$ for which $\varphi(t) \in \partial G$.
  But then, by the definition of $V$,
  \[V(\partial G) \leq V(\varphi(t)) \leq I_{-\infty,t}(\varphi) \leq I_{-\infty,0}(\varphi) = V(\bar{G}^c).\]

  To see why such a boundary regularity assumption is important, consider a punctured ball. Let $a\in E$, be such that $0<|a|_E<1$ and $V(a)<\inf_{|x|_E = 1} V(x)$. Let $G = \{x\in E: |x|<1, x \not = a\}$. We should not expect exit place results to be based on the values of $V$ on $\partial G$, because $a \in \partial G$ but $a$ is far from $\bar{G}^c$.
\end{Remark}

Now we introduce the stopping times
 \begin{equation}
   \tau^\e_x = \inf \{t>0: X^\e_x(t) \not \in G \}.
 \end{equation}
 \begin{Theorem} \label{exit-thm}
  Under Hypotheses \ref{H:linear}, \ref{H:QandZ}, \ref{H4} and either Hypothesis \ref{H-F-properties} if the equation has additive noise or Hypotheses \ref{H:B-multiplicative} and \ref{H:F-mult} if the equation has multiplicative noise, for any $x\in G$,
 \begin{enumerate}
   \item[i.]\begin{equation} \label{expected-exit-time-eq}
     \lim_{\e \to 0} \e \log \E \tau^\e_x = V(\partial G).
   \end{equation}
   \item[ii.]For any $\eta>0$,
   \begin{equation} \label{exit-time-eq}
     \lim_{\e \to 0} \Pro \left( e^{\frac{1}{\e}(V(\partial G) - \eta)} \leq \tau^\e_x \leq e^{\frac{1}{\e} (V(\partial G) + \eta)} \right) = 1.
   \end{equation}
   \item[iii.] For any closed $N \subset \partial G$ with $V(N) > V(\partial G)$,
   \begin{equation} \label{exit-place-eq}
     \lim_{\e \to 0} \Pro \left( X^\e_x(\tau^\e_x) \in N \right) = 0.
   \end{equation}
 \end{enumerate}
 \end{Theorem}

 \subsection{Some preliminary lemmas}
 \begin{Lemma} \label{et0}
   For any $\eta>0$, there exist $\rho>0$, $T_1>0$, $\delta>0$, and $\psi \in L^2([0,T_1];H)$ such that
   \[\frac{1}{2} \left| \psi \right|_{L^2([0,T_1];H)}^2 \leq V(\partial G) + \eta.\]
   and for all $|x|_E \leq \rho$,
   \[\textnormal{dist}_E(X^\psi_x(T_1), G) > \delta.\]
 \end{Lemma}

 \begin{proof}
   First, by Hypothesis \ref{H4}, there exists $y \in \bar{G}^c$ such that $V(y) < V(\partial G) + \eta$. Then, by Corollary \ref{cor:V-is-minimum}, there must exist $\varphi \in C((-\infty,0);E)$ with
   \[\lim_{t \to -\infty}|\varphi(t)|_E=0, \ \varphi(0)= y, \text{ and } I_{-\infty,0}(\varphi) < V(\partial G) + \eta.\]

   Because $y \in \bar{G}^c$, the distance, $d := \textnormal{dist}_E(y, G)$, is strictly positive. Because \\$\lim_{t\to -\infty} \varphi(t) = 0$, we can choose $T_1>0$ to be large enough so that $|\varphi(-T_1)|_E < \frac{d}{3}$. Set $x_1 = \varphi(-T_1)$.
   By \eqref{I_t,T-def}, it is clear that $\varphi(t-T_1) = X^\psi_{x_1}(t)$ for some $\psi$ with
   \[\frac{1}{2}|\psi|_{L^2([0,T_1];H)}^2 < V(\partial G) + \frac{2 \eta}{3}.\]
   Then, by \eqref{eq:init-value-continuity}, if $|x|_E  < \frac{d}{3}$,
   \[\left|y - X^\psi_x(T_1) \right|_E = \left| X^\psi_{x_1}(T_1) - X^\psi_{x}(T_1) \right|_E \leq  |x_1 -x|_E \leq \frac{2d}{3}. \]
   In particular, this means that for all $|x|_E< \frac{d}{3}$,
   \[ \textnormal{dist}_E(X^\psi_x(T_1), G) \geq \dist_E(y, G) - |y - X^\psi_x(T_1)|_E > \frac{d}{3}.\]
      Our result follows with $\rho = \frac{d}{3}$ and $\delta = \frac{d}{3}$.
 \end{proof}

 \begin{Lemma} \label{et1}
   For any $\eta>0$, we can find $T>0$ such that
   \begin{equation}
     \liminf_{\e \to 0 } \e \log \left(\inf_{ x \in G} \Pro(\tau^\e_x <T) \right) > -(V(\partial G) + \eta).
   \end{equation}
 \end{Lemma}

 \begin{proof}
   Let $\psi$, $\rho$, $\delta$, and $T_1$ satisfy Lemma \ref{et0}.
   Then, because $\dist_E(X^\psi_x(T_1),G)>\delta$, for $|x|_E< \rho$, we have the inclusion
   \[\{\tau^\e_x \leq T_1\} \supset \{|X^\e_x - X^\psi_x|_{C([0,T_1];E)} < \delta\}.\]
   By \eqref{LDP-lower}, this implies thatd
   \begin{equation}\label{less-than-rho-bound}
   \begin{array}{l}
     \ds{\lim_{\e \to 0} \e \log \left( \inf_{|x|_E \leq \rho} \Pro \left( \tau^\e_x \leq T_1 \right) \right)}\\
     \ds{\geq \lim_{\e \to 0}\e \log \left( \inf_{|x|_E \leq \rho} \Pro \left( |X^\e_x - X^\psi_x|_{C([0,T_1];E)} < \delta \right)  \right)}\\
     \ds{\geq -(V(\partial G) + \eta).}
   \end{array}
   \end{equation}
   By \eqref{eq:attraction}, there exists $T_2$ such that $\sup_{x \in G}|X^0_x(T_2)|_E < \frac{\rho}{2}$.

   By the Markov property,
   \begin{equation*}
   \begin{array}{l}
     \ds{\inf_{x \in G} \Pro \left( \tau^\e_x \leq T_1 +T_2 \right) \geq \inf_{x \in G } \Pro \left( |X^\e_x(T_2)|<\rho \text{ and } \tau^\e_{X^\e_x(T_2)} \leq T_1  \right)}\\
      \ds{\geq  \inf_{x \in G } \Pro \left(|X^\e_x - X^0_x|_{C([0,T_2])} < \frac{\rho}{2} \right)  \inf_{|x|_E < \rho } \Pro \left( \tau^\e_x \leq T_1 \right) }
   \end{array}
   \end{equation*}
   which means that, by \eqref{less-than-rho-bound},
   \[\liminf_{\e \to 0} \e \log \left( \inf_{x \in G } \Pro(\tau^\e_x \leq T_1 + T_2) \right) \geq -(V(\partial G) + \eta).\]
\end{proof}

   \begin{Lemma} \label{et2}
     Let $N \subseteq \partial G$ be closed and let $0< \nu < V(N)$. Then there exists $\rho>0$, such that for any $\varphi \in C([0,T];E)$, with $|\varphi(0)|<\rho$ and $I_{0,T}(\varphi)\leq \nu$, it holds that
     \[\inf_{t \in [0,T]}\textnormal{dist}_E(\varphi(t),N)> |\varphi(0)|_E.\]
   \end{Lemma}

   \begin{proof}
     Suppose by contradiction there exist sequences $\{T_n\} \subset \mathbb{R}$, $\{x_n\} \subset E$,\\ $\{\psi_n\} \subset L^2([0,T_n];H)$, such that
     \[\lim_{n \to 0} |x_n|_E = 0, \ \textnormal{dist}_E(X^{\psi_n}_{x_n}(T_n), N) \leq |x_n|_E, \text{ and } I_{0,T_n}(X^{\psi_n}_{x_n}) \leq \frac{1}{2} |\psi_n|_{L^2([0,T_n];H)}^2 \leq \nu.\]
       Then we define
     \[\varphi_n(t) = \left\{ \begin{array}{ll} 0 & \text{if } t \leq -T_n \\ X^{\psi_n}_0(t+T_n) & \text{if } -T_n < t \leq 0 \end{array}  \right. \]
     Notice that
     \[I_{-\infty,0}(\varphi_n) = I_{0,T_n}(X^{\psi_n}_0) \leq \nu.\]
     By Theorem \ref{thm:compact-level-sets}, $I_{-\infty,0}$ has compact level sets. Therefore, because $I_{-\infty,0}(\varphi_n) \leq \nu$ for all $n$, there exists a subsequence of $\{\varphi_n\}$ (which we relabel as $\{\varphi_n\}$) that converges to a limit $\varphi$, with $I_{-\infty,0}(\varphi) \leq \nu$.

     We also notice that by \eqref{eq:init-value-continuity},
     \[\textnormal{dist}_E(\varphi_n(0), N) \leq |X^{\psi_n}_{x_n}(T_n) - X^{\psi_n}_0(T_n)|_E + \textnormal{dist}_E(X^{\psi_n}_{x_n}(T_n),N)  \leq 2 |x_n|_E\]
     Therefore, because $N$ is closed and $|x_n|_E\to 0$, it follows that $\lim_{n \to \infty}\varphi_n(0)= \varphi(0) \in N$. This is a contradiction because $V(N) \leq I_{-\infty,0}(\varphi) \leq \nu < V(N)$.
   \end{proof}
   Let
   \begin{equation}
    \begin{array}{l}
     \ds{\Gamma_\rho = \{x \in E: |x|_E = 2\rho \},}\\
     \vs
     \ds{\gamma_\rho = \{x \in E: |x|_E \leq \rho \},}\\
     \vs
      \ds{\text{ and } \tau^\e_{1,x} = \inf\{t>0: X^\e_x(t) \in \gamma_\rho \cup \partial G\}.}
    \end{array}
   \end{equation}
   \begin{Lemma} \label{et3}
     For any $\rho>0$, such that $\gamma_\rho \subset G$,
   \begin{equation}
       \limsup_{t \to +\infty} \limsup_{\e \to 0} \e \log \left(\sup_{x \in G} \Pro \left( \tau^\e_{1,x} \geq t \right) \right) = -\infty.
     \end{equation}
   \end{Lemma}

   \begin{proof}
     By Theorem \ref{lem:attraction}, there exists $T>0$ such that for any $x\in G$, $|X^0_x(T)|_E \in \gamma_{\rho/4}$.
    This means that if $X^\psi_x$ is a controlled trajectory with the property that $X^\psi_x(t) \in G \setminus \gamma_{\rho/2}$ for all $t \in [0,T]$, then,
     \[\frac{\rho}{4} \leq \left|X^\psi_x - X^0_x \right|_{C([0,T];E)}.\]
     By Lemma \ref{lem:continuity}, and the fact that $L$ is a bounded operator, there must exist some $c :=c \left(\sup_{x \in G}|x|_E, \frac{\rho}{4} \right)>0$ such that
     \[\frac{\rho}{4} \leq \left|X^\psi_x - X^0_x \right|_{C([0,T];E)} \Longrightarrow c \leq \sup_{t \in [0,T]} \left| \int_0^t S(t-s) Q \psi(s) ds \right|_E \leq \|L\|_{\mathcal{L}(L^2((0,+\infty);H):E)} |\psi|_{L^2([0,T];E)}. \]
     From these observations, it follows that if $X^\psi_x(t) \in G \setminus \gamma_{\rho/2}$ for all $t \in [0,T]$, then
     \[I_{0,T}(X^\psi_x) = \frac{1}{2} |\psi|_{L^2([0,T];H)}^2 \geq \frac{c^2}{2\|L\|_{\mathcal{L}(L^2((0,+\infty);H):E)}}:= a>0.\]
     Another way to say this is
     \[K_{0,T}^x(a) \subseteq \{\varphi \in C([0,T];E): \varphi(t) \not \in G \setminus \gamma_{\rho/2} \text{ for some } t \in [0,T] \}\]
     where $K^x_{0,T}$ is the level set defined by \eqref{K_0T-def}.
     Then, if $\varphi$ is a trajectory such that $\varphi(t) \in G \setminus \gamma_\rho$ for all $t \in [0,T]$,
     \[\dist_{C([0,T];E)}\left(\varphi, K_{0,T}\left(a \right) \right)> \frac{\rho}{2}.\]
     Because the event
     \[\{\tau^\e_{1,x} \geq T\} \subseteq \{X^\e_x(t) \in G \setminus \gamma_\rho, \text{ for all } t \in [0,T]\} \subseteq \left\{\dist_{C([0,T];E)}\left(X^\e_x, K^x_{0,T} \left(a \right) \right)> \frac{\rho}{2}\right\},\]
     by the large deviations principle \eqref{LDP-upper},
     \begin{equation*}
     \begin{array}{l}
       \ds{\limsup_{\e \to 0} \e \log \left( \sup_{x \in G} \Pro \left( \tau^\e_{1,x} \geq T \right) \right) } \\
       \ds{\leq \limsup_{\e \to 0} \e \log \left( \sup_{x \in G} \Pro \left( \textnormal{dist}_{C([0,T];E)}\left(X^\e_x, K_{0,T}\left(a \right) \right) > \frac{\rho}{2} \right) \right) \leq -a. }
     \end{array}
     \end{equation*}

     By the Markov property, for any $k \in \nat$,
     \[\sup_{x \in G} \Pro \left( \tau^\e_{1,x} \ge kT \right) \leq \left( \sup_{x \in G} \Pro \left( \tau^\e_{1,x} \geq T \right)  \right)^k \]
     and therefore,
     \[\limsup_{\e \to 0}\e \log \left( \sup_{x \in G} \Pro \left( \tau^\e_{1,x} \geq Tk \right) \right) \leq -ka. \]
     Our result follows because we can choose $k$ to be arbitrarily large.
   \end{proof}

   \begin{Lemma} \label{et4}
     Let $N \subseteq \partial G$ be closed.
     Then,
     \begin{equation}
       \limsup_{\rho \to 0} \limsup_{\e \to 0} \e \log \left( \sup_{x \in \Gamma_\rho} \Pro\left( X^\e_x(\tau^\e_{1,x}) \in N \right) \right) \leq - V(N).
     \end{equation}
   \end{Lemma}

   \begin{proof}
     Let $\nu_0 < V(N)$. Let $\rho_0$ be the radius from Lemma \ref{et2} corresponding to $\nu_0$, and choose $ \rho < \frac{\rho_0}{2}$. Then, if $\varphi \in C([0,T];E)$ with $\varphi(0) \in \Gamma_\rho$ and $I_{0,T}( \varphi) \leq \nu_0$, it follows from Lemma \ref{et2} that
     \[\inf_{t \in [0,T]}\textnormal{dist}_E(\varphi(t),N) >|\varphi(0)|_E = 2\rho.\]
     Therefore, for any $T>0$, we have
     \[\begin{array}{l}
     \ds{\left\{ X^\e_x(\tau^\e_{1,x}) \in N, \tau^\e_{1,x} \leq T \right\} \subset \left\{X^\e_x(t) \in N \text{ for some } t \leq T \right\}}\\
     \vs
     \ds{\subset \left\{\dist_{C([0,T];E)}(X^\e_x, K^x_{0,T}(\nu_0)) > 2 \rho \right\}.}
     \end{array}\]
     By the large deviations principle \eqref{LDP-upper}, for any $T>0$, this implies that
     \[\begin{array}{l}
     \ds{\limsup_{\e\to 0} \e \log  \left(\sup_{x \in \Gamma_\rho} \Pro \left( X^\e_x(\tau^\e_{1,x}) \in N, \tau^\e_{1,x} \leq T \right)   \right)}\\
     \ds{ \leq
       \limsup_{\e \to 0} \e \log \left( \sup_{x \in \Gamma_\rho} \Pro \left( \textnormal{dist}_{C([0,T];E)} \left( X^\e_x, K^x_{0,T}(\nu_0)  \right) > 2  \rho \right) \right)}\\
       \ds{ \leq -\nu_0.} \end{array}\]
     Furthermore, by Lemma \ref{et3}, we can find $T$ large enough so that
     \[\limsup_{\e \to 0} \e \log \left( \sup_{x \in \Gamma_\rho} \Pro \left( \tau^\e_{1,x} \geq T \right) \right) \leq - \nu_0.\]
     Then we observe that
     \[\Pro \left( X^\e_x(\tau^\e_{1,x}) \in N \right) \leq \Pro\left( X^\e_x(\tau^\e_{1,x}) \in N, \tau^\e_{1,x} \leq T \right) + \Pro\left( \tau^\e_{1,x} \geq T \right) \]
     and therefore,
     \[\limsup_{\e \to 0} \e \log \left( \sup_{x \in \Gamma_\rho} \Pro \left( X^\e_x(\tau^\e_{1,x}) \in N \right) \right) \leq -\nu_0.\]
     The result follows because $\nu_0< V(N)$ was arbitrary.

   \end{proof}

   \begin{Lemma} \label{et6}
     For any $\rho>0$ such that $\gamma_\rho \subset G$, and any $x \in G$,
     \begin{equation}
       \lim_{\e \to 0} \Pro \left( X^\e_x(\tau^\e_{1,x}) \in \gamma_\rho \right) = 1.
     \end{equation}
   \end{Lemma}

   \begin{proof}
     Fix $x \in G$. Let $\rho_0 < \min\{ \frac{\rho}{2}, \inf_{t >0}\dist_E(X^0_x(t),\bar{G}^c)\}$, where $X^0_x$ be the unperturbed trajectory starting at $x \in G$. Because $X^0_x \to 0$, we can find $T>0$ such that $|X^0_x(T)|_E < \frac{\rho}{2}$.
     Then, 
     \[\Pro \left( X^\e_x(\tau^\e_{1,x}) \in \gamma_\rho \right) \geq \Pro \left( \left|X^\e_x - X^0_x \right|_{C([0,T];E)}< \rho_0 \right) \to 1.\]
   \end{proof}

   \begin{Lemma} \label{et5} For fixed $\rho>0$,
     \[\lim_{T \to 0} \limsup_{\e \to 0} \e \log \left( \sup_{x \in \gamma_\rho} \Pro \left( \text{There exists } t \in [0,T], X^\e_x(t) \in \Gamma_\rho \right) \right) = -\infty.\]
   \end{Lemma}

   \begin{proof}
     Since $|X^0_x(t)|_E \leq |x|_E$ for all $t \geq 0$, we have
     \[\sup_{t \geq0} \sup_{x \in \gamma_\rho} |X^0_x(t)|_E \leq \rho.\]
     We define
     \[\Phi_T:= \left\{\varphi \in C([0,T];E): \varphi(0) \in \gamma_\rho, \text{ and there exists } t \in [0,T] \text{ such that } |\varphi(t)|_E = \frac{3 \rho}{2}  \right\}\]
     and
     \[a(T) := \inf_{\varphi \in \Phi_T} I_{0,T}(\varphi).\]
     Then, because $\Gamma_\rho = \{x  \in E: |x|_E = 2 \rho\}$, it follows from \eqref{LDP-upper} that
     \begin{equation*}
      \begin{array}{l}
       \ds{\limsup_{\e \to 0} \e \log \left(\sup_{x \in \gamma_\rho} \Pro \left( \text{ there exists }t \in [0,T] \text{ such that }  X^\e_x(t) \in \Gamma_\rho \right) \right)}\\
       \ds{\leq \limsup_{\e \to 0} \e \log \left( \sup_{x \in \gamma_\rho} \Pro \left( \textnormal{dist}_{C([0,T];E)} (X^\e_x, K^x_{0,T}(a(T))) > \frac{  \rho}{2} \right)  \right)
       \leq -a(T)}.
      \end{array}
     \end{equation*}

     Therefore, it remains to show that
     \begin{equation} \label{eq:a->infty}
       \lim_{T \to 0} a(T)  = +\infty.
     \end{equation}
     Let $\{T_n\}$ be a sequence such that $T_n \to 0$ and for any $n \in \nat$ let $x_n \in \gamma_\rho$ and $\psi_n \in L^2([0,T_n];H)$ such that $X^{\psi_n}_{x_n} \in \Phi_{T_n}$ and $\frac{1}{2}|\psi_n|_{L^2([0,T_n];H)}^2 \leq a(T_n) + \frac{1}{n}$. If we show that $|\psi_n|_{L^2([0,T_n];H)} \to +\infty$, then \eqref{eq:a->infty} follows.

      Because $X^{\psi_n}_{x_n} \in \Phi_{T_n}$,
     \[\frac{3  \rho}{2} \leq |X^{\psi_n}_{x_n}|_{C([0,T_n];E)} \leq |X^{\psi_n}_{x_n} - X^0_{x_n}|_{C([0,T_n];E)} + |X^0_{x_n}|_{C([0,T];E)} \leq |X^{\psi_n}_{x_n} - X^0_{x_n}|_{C([0,T_n];E)} + \rho.\]
     Therefore,
     \[\frac{\rho}{2} \leq |X^{\psi_n}_{x_n} - X^0_{x_n}|_{C([0,T_n];E)} .\]
     By Lemma \ref{lem:continuity}, there must exist some $c>0$ independent of $n$ such that
     \[\frac{\rho}{2} \leq |X^{\psi_n}_{x_n} - X^0_{x_n}|_{C([0,T_n];E)} \Longrightarrow c\leq \sup_{t \leq T_n}\left|\int_0^t S(t-s) Q \psi_n(s) ds \right|_E\leq \|L_t\|_{\mathcal{L}(L^2([0,t];H),E)}|\psi_n|_{L^2([0,t];H)}.\]
     But by Corollary \ref{cor:L_t->0}, $\|L_t\|_{\mathcal{L}(L^2([0,t];H),E)} \to 0$.
    Therefore we conclude that
    \[\lim_{n \to +\infty} |\psi_n|_{L^2([0,T_n];H)} = +\infty\]
    and \eqref{eq:a->infty} follows.
   \end{proof}
\subsection{Proof of Theorem \ref{exit-thm}}
\begin{proof}
  The following proofs are based on the arguments used in \cite{dz}. For completeness, they are included below.\\
  \textbf{Upper Bound}

  By the Markov property, for fixed $\e>0$ and $T>0$,
  \[ \sup_{x \in G} \Pro \left( \tau^\e_x \geq kT \right) \leq \left( \sup_{x \in G} \Pro \left(\tau^\e_x \geq T \right) \right)^k.\]
  It follows that
  \[\begin{array}{l}
  \ds{\E (\tau^\e_x) \leq T \sum_{k=0}^\infty \Pro \left( \tau^\e_x \geq kT \right) \leq T \sum_{k=0}^\infty \left( \sup_{x \in G} \Pro \left(\tau^\e_x \geq T \right) \right)^k}\\
  \ds{\leq \frac{T}{1 - \left( \sup_{x \in G} \Pro \left(\tau^\e_x \geq T \right) \right)} \leq \frac{T}{\inf_{x \in G} \Pro \left( \tau^\e_x < T \right)}.} \end{array}  \]
  Fix $\eta>0$.  By Lemma \ref{et1}, we can find a $T>0$ such that
  \[\liminf_{\e \to 0} \e \log \left( \inf_{x \in G} \Pro \left( \tau^\e_x < T \right) \right) > -\left(V(\partial G) + \frac{\eta}{2} \right) .\]
  Therefore,
  \[\limsup_{\e \to 0} \e \log \E \tau^\e_x \leq \limsup_{\e \to 0} \e \log \left( \frac{T}{\inf_{x \in G} \Pro \left( \tau^\e_x < T \right)} \right) \leq (V(\partial G) + \eta)\]
  and because $\eta >0$ was arbitrary, the upper bound of \eqref{expected-exit-time-eq} follows. The upper bound of \eqref{exit-time-eq} follows by a straightforward application of the Chebyshev inequality.\\
  \textbf{Lower Bound}

  Let $\gamma_\rho = \{x \in E: |x|_E \leq \rho\}$ and $\Gamma_\rho = \{x \in E: |x|_E = 2 \rho\}$.
  Define the stopping times
  \begin{equation} \label{stopping-times}
    \begin{array}{l}
      \tau^\e_{1,x} = \inf\{t>0: X^\e_x(t) \in \gamma_\rho \cup \partial G\}\\
      \sigma^\e_{n+1,x} = \inf\{t> \tau^\e_{n,x}: X^\e_x(t) \in \Gamma_\rho\}\\
      \tau^\e_{n+1,x} = \inf\{t> \sigma^\e_{n+1,x}: X^\e_x(t) \in \gamma_\rho \cup \partial G\}.
    \end{array}
  \end{equation}
  Let $\eta>0$. Using Lemma \ref{et4}, find $\rho >0$ small enough so that
  \[\limsup_{\e \to 0} \e \log \left( \sup_{x \in \Gamma_\rho} \Pro \left( X^\e_x(\tau^\e_{1,x}) \in \partial G \right) \right) < - V(\partial G) + \frac{\eta}{2}. \]
  Notice that for $m \geq 2$, by the Markov property,
  \[\begin{array}{l}
   \ds{\sup_{x \in G} \Pro \left( \tau^\e_x = \tau^\e_{m,x} \right) \leq \left( \sup_{x \in G} \Pro \left( \tau^\e \not =  \tau^\e_{k,x}, 1\leq k \leq m-1  \right) \right) \left(\sup_{x \in \Gamma_\rho} \Pro \left( \tau^\e_{1,x} \in \partial G \right) \right)}\\
    \vs
    \ds{\leq \sup_{x \in \Gamma_\rho} \Pro \left( \tau^\e_{1,x} \in \partial G \right).}
    \end{array}  \]
  Next, using Lemma \ref{et5}, we find $T_0>0$ small enough so that
  \[\limsup_{\e \to 0} \e \log \left( \sup_{x \in \gamma_\rho} \Pro \left( \text{There exists } t\in [0,T_0], X^\e_x(t) \in \Gamma_\rho \right) \right) \leq -V(\partial G).\]
  A consequence of this is that for any $n \in \nat$
  \[\limsup_{\e \to 0} \e \log \sup_{x \in G} \P \left( \sigma^\e_{n+1,x} - \tau^\e_{n,x} \leq T_0 \right) \leq -V(\partial G).\]
  In \cite{dz}, they observe that the event $\{\tau^\e_x \leq kT_0\}$ implies either that $\{ \tau^\e_x = \tau^\e_{m,x}\}$ for some $m \leq k+1$ or that at least one of the excursion times, $\tau^\e_{m+1,x} - \tau^\e_{m,x} \leq T_0$. Therefore, for any $k \in \nat$, $x \in G$, and small enough $\e$,
  \[\Pro \left( \tau^\e_x \leq k T_0 \right) \leq \sum_{m=1}^{k + 1} \left( \Pro\left( \tau^\e_x = \tau^\e_{m,x} \right) + \Pro \left( \sigma^\e_{n+1,x} - \tau^\e_{n,x} \leq T_0 \right) \right) \leq \Pro \left(\tau^\e_x = \tau^\e_{1,x} \right) + 2k e^{-\frac{1}{\e}(V(\partial G) - \frac{\eta}{2})}\]
  If we set $k = \left[\frac{e^{\frac{1}{\e}(V(\partial G) - \eta)}}{T_0} \right] + 1$ then we see that for small enough $\e>0$,
  \[\Pro \left( \tau^\e_x \leq e^{\frac{1}{\e} (V(\partial G) - \eta)} \right) \leq  \Pro \left( \tau^\e_{1,x} \in \partial G \right) + \frac{4}{T_0} e^{-\frac{\eta}{2\e}}.\]
  We apply Lemma \ref{et6} to see that the above quantity converges to $0$ as $\e \to 0$. We have proven the lower bound for \eqref{exit-time-eq}.  The lower bound for \eqref{expected-exit-time-eq} follows by a straightforward application of the Chebyshev inequality.\\
  \textbf{Exit Place}

  The proof of the exit place is very similar to the proof of the lower bound of the exit time. Let $N \subset \partial G$ be closed and have the property that $V(N)> V(\partial G)$. Let $0<\eta< \frac{1}{3}(V(N) - V(\partial G))$. Using Lemma \ref{et4}, we find $\rho$ small enough so that
  \[\limsup_{\e \to 0} \e \log \left(\sup_{x \in \Gamma_\rho} \Pro \left(X^\e_x(\tau^\e_{1,x}) \in N \right) \right) \leq - \left(V(N) - \frac{\eta}{2} \right).\]
  Next, using Lemma \ref{et5}, we choose $T_0$ small enough so that
  \[\limsup_{\e \to 0} \e \log \left( \sup_{x \in \gamma_\rho} \Pro \left( \textnormal{There exists } t\in [0,T_0], X^\e_x(t) \in \Gamma_\rho \right) \right) \leq -\left(V(N)- \frac{\eta}{2} \right).\]
  Using the same stopping times defined in \eqref{stopping-times}, we observe that for $x \in G$, for $k$ to be chosen later, and for small enough $\e$,
  \begin{equation*}
    \begin{array}{l}
      \ds{\Pro \left(X^\e_x(\tau^\e_x) \in N \right) \leq \Pro \left( \tau^\e_x > \tau^\e_{k,x} \right) + \sum_{m=1}^k \Pro \left(\tau^\e_x \geq \tau^\e_{x,m} \right) \Pro \left( X^\e_x(\tau^\e_{m,x}) \in N | \tau^\e_x \geq \tau^\e_{m,x} \right)     }\\
      \ds{\leq \Pro \left( \tau^\e_x > (k-1) T_0 \right) + \Pro \left( \tau^\e_{k,x} \leq (k-1) T_0  \right) + \Pro \left( X^\e_x(\tau^\e_{1,x}) \in N \right)+ \sum_{m=2}^k \sup_{y \in \Gamma_\rho} \Pro \left( X^\e_y(\tau^\e_{1,y}) \in N \right)     }\\
      \ds{\leq \Pro \left( \tau^\e_x > (k-1) T_0 \right) + 2\Pro \left(X^\e_x(\tau^\e_{1,x}) \in N \right)+ 2k e^{-\frac{1}{\e}(V(N) -  \eta)}.         }
    \end{array}
  \end{equation*}
  In the last line, we used the fact that
  \[\Pro\left(\tau^\e_{k,x} \leq (k-1) T_0 \right) \leq  \sum_{m=2}^k \sup_{x \in \gamma_\rho} \Pro \left(\text{There exists } t \in [0,T_0], X^\e_x(t) \in \Gamma_\rho \right).\]
  This is because if $\tau^\e_{k,x} \leq kT_0$, then at least one of $\tau^\e_{m,x} - \tau^\e_{m-1,x}$, $m=2..k$ must be less than $T_0$
  By \eqref{expected-exit-time-eq} and Chebyshev inequality, for small enough $\e$,
  \[\Pro \left( \tau^\e_x > (k-1) T_0 \right) \leq \frac{e^{\frac{1}{\e}(V(\partial G) + \eta)}}{(k-1)T_0}. \]
  Then, we choose $k = [e^{\frac{1}{\e}(V(\partial G) + 2\eta)}]$. Because $V(N) - V(\partial G)> 3\eta$,
  \[2ke^{-\frac{1}{\e}(V(N) - \eta)} \leq 2e^{\frac{1}{\e}(V(\partial G) - V(N) + 3\eta)} \to 0\]
  and we are left with
  \[\limsup_{\e \to 0} \Pro \left( X^\e_x(\tau^\e_x) \in N \right) \leq 2 \limsup_{\e \to 0} \Pro \left( X^\e_x(\tau^\e_{1,x}) \in N \right).\]
  Finally, by Lemma \ref{et6},
  \[\lim_{\e \to 0} \Pro \left( X^\e_x(\tau^\e_{1,x}) \in N \right) =0\]
  and our result follows.
\end{proof}

\end{document}